\documentclass{amsart}

\usepackage{amsmath,amssymb,amsthm,mathtools}
\usepackage{mathrsfs}
\usepackage{amscd}
\usepackage{verbatim}
\usepackage{stmaryrd}
\usepackage{mathabx}

\usepackage[active]{srcltx}
\usepackage[shortlabels]{enumitem}

\mathtoolsset{centercolon}

\let\le\leqslant
\let\leq\leqslant
\let\ge\geqslant
\let\geq\geqslant

\theoremstyle{plain}
\newtheorem{theorem}{Theorem}[section]
\theoremstyle{remark}
\newtheorem{remark}[theorem]{Remark}

\theoremstyle{plain}
\newtheorem{corollary}[theorem]{Corollary}
\newtheorem{lemma}[theorem]{Lemma}
\newtheorem{proposition}[theorem]{Proposition}
\newtheorem{definition}[theorem]{Definition}

\numberwithin{equation}{section}

\newcommand{\C}{\mathbb C}
\newcommand{\R}{\mathbb R}
\newcommand{\N}{\mathbb N}
\newcommand{\Z}{\mathbb Z}

\newcommand{\E}{\mathbb E}
\newcommand{\1}{\mathbf 1}
\newcommand{\dd}{\,{\rm d}}

\newcommand{\Rad}{\operatorname{Rad}}
\newcommand{\norm}[1]{\|#1\|}

\newcommand{\wt}{\widetilde}
\newcommand{\Dom}{\mathsf{D}}

\newcommand{\mc}{\mathcal}
\newcommand{\ms}{\mathscr}

\DeclarePairedDelimiter\cbrace\{\}

\DeclarePairedDelimiter{\ip}\langle\rangle
\DeclarePairedDelimiter{\nrm}\lVert\rVert

\newcommand{\brb}[1]{\bigl(#1\bigr)}
\newcommand{\cbraceb}[1]{\bigl\{#1\bigr\}}

\newcommand{\bracb}[1]{\bigl[#1\bigr]}

\newcommand{\nrmB}[1]{\Bigl\|#1\Bigr\|}
\newcommand{\absB}[1]{\Bigl|#1\Bigr|}
\newcommand{\brB}[1]{\Bigl(#1\Bigr)}
\newcommand{\cbraceB}[1]{\Bigl\{#1\Bigr\}}

\DeclareMathOperator{\spn}{span}
\DeclareMathOperator{\loc}{loc}

\DeclareMathOperator{\supp}{supp}

\renewcommand{\P}{\mathbb{P}}

\allowdisplaybreaks

\DeclareFontFamily{U}{mathx}{\hyphenchar\font45}
\DeclareFontShape{U}{mathx}{m}{n}{<5> <6> <7> <8> <9> <10> <10.95> <12> <14.4> <17.28> <20.74> <24.88> mathx10}{}
\DeclareSymbolFont{mathx}{U}{mathx}{m}{n}
\DeclareFontSubstitution{U}{mathx}{m}{n}
\DeclareMathAccent{\widecheck}{0}{mathx}{"71}

\usepackage[colorlinks,linkcolor={blue},citecolor={blue},urlcolor={purple},]{hyperref}
\usepackage[colorinlistoftodos,prependcaption]{todonotes}

\title[Necessary conditions for maximal regularity]
{Necessary conditions for deterministic and stochastic maximal regularity}

\author{Emiel Lorist, Jan van Neerven,  Mark Veraar and Lutz Weis}

\address[Emiel Lorist, Jan van Neerven and  Mark Veraar]{\hfill\break\indent
Delft Institute of Applied Mathematics \hfill\break\indent
Delft University of Technology \hfill\break\indent
P.O. Box 5031 \hfill\break\indent
2600 GA Delft, The Netherlands}
\email{e.lorist@tudelft.nl}
\email{j.m.a.m.vanneerven@tudelft.nl}
\email{m.c.veraar@tudelft.nl}

\address[Lutz Weis]{\hfill\break\indent
Institute for Analysis \hfill\break\indent
Department of Mathematics \hfill\break\indent
Karlsruhe Institute of Technology (KIT) \hfill\break\indent
Englerstra{\ss}e 2 \hfill\break\indent
76131 Karlsruhe, Germany}
\email{lutz.weis@kit.edu}

\thanks{The first and third authors have received funding from the Veni subsidy \href{https://doi.org/10.61686/ZGRMR99948}{VI.Veni.242.057} and Vici subsidy \href{https://www.nwo.nl/en/projects/vic212027}{VI.C.212.027}, respectively, of the Netherlands Organisation for Scientific Research (NWO)}

\begin{document}

\begin{abstract}
We study the role of Banach space geometry in deterministic and stochastic maximal regularity. We first construct an example showing that the UMD assumption in the characterisation of maximal $L^p$-regularity in terms of $R$-sectoriality cannot be omitted. Combining this construction with an equivalence between stochastic maximal regularity and deterministic maximal regularity on the $2$-concavification of the underlying space, we obtain an operator on a UMD Banach function space of type $2$ that has a bounded $H^\infty$-calculus of angle zero, but fails stochastic maximal $L^p$-regularity (SMR$_p$) for every $p\in[2,\infty)$.

Motivated by this example, we study the Banach space geometry hypothesis underlying SMR$_p$ more closely. This is an $R$-boundedness condition $(S_p)$ for stochastic convolution operators. For UMD spaces $X$ of type $2$, we show that this condition is not only sufficient, but also necessary for two canonical operators: a diagonal multiplier on a Rademacher space and, for $q>2$, the Laplacian on $L^q(\mathbb R^d;X)$. Finally, we prove that its interval-kernel and exponential-kernel formulations are equivalent and that, at the endpoint $p=2$, condition $(S_2)$ holds if and only if $X$ is isomorphic to a Hilbert space.
\end{abstract}



\maketitle

\section{Introduction}

Maximal regularity lies at the intersection of evolution equations,
operator theory, and Banach space geometry. For the deterministic Cauchy
problem
\[
        u'(t)+Au(t)=f(t),\qquad u(0)=0,
\]
maximal $L^p$-regularity asserts that
\[
     f\in L^p(I;X) \quad\Longrightarrow\quad  u',Au\in L^p(I;X).
\]
For the stochastic problem
\[
        \dd U(t)+AU(t)\dd t=G(t)\dd W(t),\qquad U(0)=0,
\]
stochastic maximal $L^p$-regularity asserts, in its basic form, that
\[
        G\in L^p(\Omega\times I;X)
        \quad\Longrightarrow\quad
        A^{1/2}U\in L^p(\Omega\times I;X).
\]
Both properties are fundamental tools for treating nonlinear evolution
equations by perturbation and fixed-point methods; see
\cite{HNVW3,pruss2016moving} for the deterministic theory and
\cite{AVsurvey} for the stochastic theory.

In the deterministic setting, the role of the geometry of $X$ is described in \cite{Weis-MathAnn}: if $X$ is a UMD space and $p\in(1,\infty)$, then maximal $L^p$-regularity is equivalent to $R$-sectoriality of $A$ with angle strictly smaller than $\frac{\pi}{2}$. Classical counterexamples of Kalton--Lancien and Fackler
\cite{KaLa,KaLa02,Fack14,Fackl16} show that sectoriality is not sufficient for maximal regularity. In those examples, the obstruction can be traced to the failure of
$R$-sectoriality. This raises the question whether
$R$-sectoriality, or even a bounded $H^\infty$-calculus, might still
imply maximal regularity without the UMD assumption.

The stochastic theory has a related, but more subtle, geometric dependence. On $L^q$-spaces with $q\in[2,\infty)$, a bounded $H^\infty$-calculus of angle less than $\frac{\pi}{2}$ implies stochastic maximal $L^p$-regularity for $p\in(2,\infty)$; when
$q=2$, the endpoint $p=2$ is included. The proof developed in \cite{NVW-SMR,NVW-Rbounded,NVW-survey} extends to more general UMD Banach spaces of type $2$, provided a certain family of elementary stochastic convolution operators is $R$-bounded. It was left open whether this additional condition on $X$ merely reflects a limitation of the proof, or whether stochastic maximal regularity itself may fail on UMD spaces of type $2$, even for operators with a bounded $H^\infty$-calculus.

The purpose of this paper is to identify necessary geometric conditions in both the deterministic and stochastic settings. We will show that even very simple diagonal operators detect the relevant Banach space geometry sharply. Indeed, define  the mixed-norm sequence spaces
\[
        X_{q,r}:=\ell^q(\ell^r),\qquad q,r\in[1,\infty),
\]
and the positive diagonal operator
\begin{equation}\label{eq:defAintro}
     (Ax)_{m,n}=2^n x_{m,n}.
\end{equation}
On every $X_{q,r}$, this operator is $R$-sectorial and has a bounded $H^\infty$-calculus of angle $0$. Furthermore,  for $p,q,r\in[1,\infty)$, the operator $A$ has maximal $L^p$-regularity if and only if
\[
\left[
        p>1
        \quad\text{and}\quad
        (q>1\ \text{or}\ r=1)
\right]
\qquad\text{or}\qquad
        p=q=r=1;
\]
see Corollary \ref{cor:finaldet}. In particular, if $r>1$, then the operator on $\ell^1(\ell^r)$ fails maximal $L^p$-regularity for every $p\in[1,\infty]$, despite being positive, $R$-sectorial, and having a bounded $H^\infty$-calculus of angle $0$. Thus the UMD assumption in the characterisation of deterministic maximal regularity cannot simply be omitted, even for elementary diagonal multipliers on Banach function spaces.
A more general form of the construction shows that for the diagonal multiplier on $X(\ell^2)$, where $X$ is a Banach function space, we have 
\[
\begin{aligned}
 A\text{ is $R$-sectorial on }X(\ell^2)
     &\quad\Longleftrightarrow\quad
     X\text{ has finite cotype},\\
 A\text{ has maximal $L^p$-regularity on }X(\ell^2)
     &\quad\Longleftrightarrow\quad
     X\text{ is UMD};
\end{aligned}
\]
see Theorem \ref{thm:Xexample}. Choosing a non-UMD Banach function space $X$ for which both $X$ and $X^*$ have finite cotype thus produces an example in which both $A$ and $A^*$ are $R$-sectorial and have bounded $H^\infty$-calculi of angle $0$, but neither has maximal $L^p$-regularity; see Corollary \ref{cor:dual-stable-counterexample}.

The same diagonal multiplier also provides the bridge from deterministic
to stochastic maximal regularity. For $p,q,r\in[2,\infty)$, we prove
the equivalence
\[
\begin{gathered}
A\text{ has stochastic maximal $L^p$-regularity on }
\ell^q(\ell^r)
\\
\Longleftrightarrow
\\
A\text{ has deterministic maximal $L^{p/2}$-regularity on }
\ell^{q/2}(\ell^{r/2})
\end{gathered}
\]
see Proposition \ref{prop:MRSMR}. Combining this with the deterministic result gives a complete stochastic classification: $A$ has stochastic maximal $L^p$-regularity on $X_{q,r}$ if and only if
\[
\left[
        p>2
        \quad\text{and}\quad
        (q>2\ \text{or}\ r=2)
\right]
\qquad\text{or}\qquad
        p=q=r=2;
\]
see Theorem \ref{thm:mainSMR}. In particular, for every $r>2$, the space
\[
        X=\ell^2(\ell^r)
\]
is a UMD Banach function space of type $2$, and the operator in \eqref{eq:defAintro} has a bounded $H^\infty$-calculus of angle $0$, but fails stochastic maximal $L^p$-regularity for every $p\in[2,\infty)$; see Corollary \ref{cor:main}. The choice $r=4$ gives precisely the space $\ell^2(\ell^4)$ for which the method of \cite{NVW-Rbounded} was known to break down. Our result shows that this obstruction is not merely a limitation of that method: stochastic maximal regularity itself fails.

\medskip

Motivated by this example, we next investigate the geometric condition underlying the positive stochastic maximal regularity theory. Let $X$ be a UMD Banach space of type $2$. For $\delta>0$, let
\[
        h_\delta(t):=\delta^{-1/2}\mathbf 1_{(0,\delta)}(t),
\]
and for a strongly progressively measurable $G \in L^p(\Omega\times I;X)$ define the stochastic convolution operator 
\[
        (S_{h_\delta}G)(t)
        :=
        \int_0^t h_\delta(t-s)G(s)\dd W(s).
\]
Then $X$ is said to satisfy condition
$(S_p)$ if the family
\[
        \{S_{h_\delta}:\delta>0\}
\]
is $R$-bounded. We also consider the exponential kernels
\[
        k_\lambda(t):=\lambda^{1/2}e^{-\lambda t},
        \qquad \lambda>0,
\]
and the corresponding condition $(S_p^{\exp})$. The exponential formulation is the one naturally associated with analytic semigroups, whereas the interval formulation is better suited to geometric arguments.

Our first structural result is that the two formulations are equivalent:
\[
        (S_p)\quad\Longleftrightarrow\quad(S_p^{\exp});
\]
see Proposition \ref{prop:HpEp}. The difficult implication is from exponential kernels to interval kernels. Its proof uses the triangular contraction property, which UMD spaces always have, together with a bounded variation contraction principle. As a consequence, from the independence of $p\in(2,\infty)$ for $(S_p^{\exp})$, proven in \cite{LoVer}, we deduce that $(S_p)$ is independent of $p\in(2,\infty)$ as well; see Corollary \ref{cor:Sp-independent-p}. The endpoint behaves very differently:
\[
        X\text{ satisfies }(S_2)        \quad\Longleftrightarrow\quad        X\text{ is isomorphic to a Hilbert space};
\]
see Theorem \ref{thm:HScharacS2}. 

To conclude the paper, we show that condition $(S_p)$ is not merely a technical hypothesis for stochastic maximal regularity. More precisely, it is forced by two canonical operators. First, let $A$ be the diagonal multiplier
\[
        A(x_n)_{n\geq1}=(2^nx_n)_{n\geq1}
\]
on the Rademacher space $\Rad^p(X)$ for $p \in [2,\infty)$. This operator always has a bounded $H^\infty$-calculus of angle $0$, and in Theorem \ref{thm:SMRRad} we prove that
\begin{equation*}
\begin{split}
&A\text{ has stochastic maximal $L^p$-regularity}\quad
\Longleftrightarrow\quad
X\text{ satisfies }(S_p).
\end{split}
\end{equation*}
Combining this with a similar result for the Laplacian in Theorem \ref{thm:large-scale-plane-wave-extraction}, we obtain the following equivalences for $p,q\in(2,\infty)$:
\begin{enumerate}[label={\rm(\roman*)}, leftmargin=*]
\item $X$ satisfies condition $(S_p)$;
\item $A$ has stochastic maximal $L^p$-regularity on $\Rad^p(X)$;
\item $-\Delta$ has stochastic maximal $L^p$-regularity on
$L^q(\R^d;X)$.
\end{enumerate}

\medskip

The paper is organised as follows. After the preliminaries in Section \ref{sec:prelim}, Section \ref{sec:counter} studies the diagonal multiplier on mixed-norm sequence spaces. We first establish the complete deterministic classification, then pass to stochastic maximal regularity by $2$-concavification, and finally derive the general deterministic consequences on $X(\ell^2)$. In Section \ref{sec:Rbdd} we study conditions $(S_p)$ and $(S_p^{\exp})$, prove their equivalence and the Hilbert space characterisation at $p=2$. Moreover, we establish the characterisations in terms of the diagonal multiplier on $\Rad^p(X)$ and the Laplacian on $L^q(\R^d;X)$.

Notation and terminology are standard and can be found in \cite{HNVW1,HNVW2}. For an  overview of deterministic maximal $L^p$-regularity, we refer to \cite[Chapter 17]{HNVW3}. 

\section{Preliminaries}\label{sec:prelim}
Throughout the paper, all Banach spaces are over the complex scalar field. Rademacher random variables are always regarded as complex-valued, that is, as Steinhaus variables; cf. the conventions of \cite{HNVW2}. Moreover, $A$ denotes a densely defined sectorial operator on a Banach space $X$, and $(S(t))_{t\geq0}$ denotes the analytic $C_0$-semigroup generated by $-A$. Whenever fractional powers are used, $A$ is assumed to be injective. We refer to \cite[Chapter 13]{Nee} and \cite[Chapter 10]{HNVW2} for the relevant definitions and conventions.
For details on $R$-boundedness, $R$-sectoriality and the $H^\infty$-functional calculus, the reader is referred to \cite{HNVW2} as well. 

\subsection{Deterministic  maximal regularity} Let $I$ be either $(0,\infty)$ or a bounded interval $(0,T)$ and consider the following problem on  $I$:
\begin{equation}
\label{eq:EEMR}
\left\{
\begin{aligned}
u' + A u  &= f\,\\
u(0)&=0,
\end{aligned}
\right.
\end{equation}
where $f\in L^1_{\loc}(\overline{I};X)$ is given. The solution \eqref{eq:EEMR} can be written by the {\em mild solution formula}
\[u_f(t) := \int_0^t S(t-s) f(s) \dd s, \qquad t \in I.\]
For a given $p\in [1, \infty]$, the operator $A$ is said to have {\em maximal $L^p$-regularity on $I$} if
there exists a constant $C$ such that for all $f\in L^p(I;X)$, one has
$u_f(t)\in \Dom(A)$ for almost all $t\in I$ and
\[\|A u_f\|_{L^p(I;X)}\leq C \|f\|_{L^p(I;X)}.\]
Maximal $L^p$-regularity has various permanence properties with respect to $p$ and $I$. 

If $X$ is a UMD space, and $p\in(1,\infty)$, \cite{Weis-MathAnn} states that maximal
$L^p$-regularity on bounded intervals is equivalent to
$R$-sectoriality of $A$ with angle less than $\frac12\pi$. 
Under the additional assumption $0\in\varrho(A)$, the corresponding
characterisation holds on $\R_+$.
Maximal $L^p$-regularity provides an effective tool in the study of nonlinear evolution equations. The reader is referred to \cite[Chapters 17, 18]{HNVW3} and \cite{pruss2016moving} for further details.

\subsection{Stochastic maximal regularity}
Let $X$ be a UMD Banach space, and let $(\Omega,\mathcal A,(\mathcal F_t)_{t\geq0},\P)$ be a filtered probability space carrying a real $(\mathcal F_t)_{t\geq0}$-Brownian motion $W$. For an introduction to  stochastic integration in UMD spaces, the reader is referred to \cite{NVWco,NVW-UMD,NVW-survey}. The spaces $\gamma(I;X)$ are discussed in \cite[Chapter 9]{HNVW2}.

Let $p\in(1,\infty)$.
For every strongly progressively measurable process $\Phi$ representing an element of $L^p(\Omega;\gamma(I;X))$, the It\^o isomorphism gives
\begin{align}\label{eq:Ito}
\E\Big\|\int_I\Phi\dd W\Big\|^p
\eqsim_{p,X}
\E\|\Phi\|_{\gamma(I;X)}^p
\end{align}
whenever the right-hand side is finite. If, in addition, $X$ is a Banach function space, then
\begin{align}\label{eq:Ito2}
\E\Big\|\int_I\Phi\dd W\Big\|^p
\eqsim_{p,X}
\E\Big\|\Big(\int_I|\Phi(t)|^2\dd t\Big)^{1/2}\Big\|_X^p.
\end{align}
For the remainder of this subsection, we assume in addition that $X$ has type $2$. In this case every strongly progressively measurable process in $L^2(\Omega\times I;X)$ is stochastically integrable.

With these preparations we turn to the definition of stochastic maximal regularity. We will additionally assume $A$ is injective, which in certain cases can be avoided by shifting, i.e. considering $\lambda+A$. On $I$ consider the following problem
\begin{equation}
\label{eq:SEEMR}
\left\{
\begin{aligned}
\dd U + A U \dd t &= G \dd W,\\
U(0)&=0,
\end{aligned}
\right.
\end{equation}
where $G\in L^2_{\loc}(\Omega\times I;X)$ is strongly progressively measurable.
The solution \eqref{eq:SEEMR} can again be written by the {\em stochastic  mild solution formula}
\[U_G(t) := \int_0^t S(t-s) G(s) \dd W(s),\qquad t \in I.\]
For a given $p\in [2, \infty)$, the operator $A$ is said to have {\em stochastic maximal $L^p$-regularity on $I$} if there exists a constant $C$ such that for all strongly progressively measurable $G\in L^p(\Omega\times I;X)$, one has $U_G(t,\omega)\in \Dom(A^{1/2})$ for $(\P\otimes\dd t)$-almost all $(\omega,t)\in\Omega\times I$, and
\begin{align}\label{eq:SMRest}
\|A^{1/2} U_G\|_{L^p(\Omega\times I;X)}\leq C \|G\|_{L^p(\Omega\times I;X)}.
\end{align}

From the It\^o isomorphism \eqref{eq:Ito}, we can deduce the following
equivalent deterministic formulation of stochastic maximal regularity: The operator $A$ has stochastic maximal $L^p$-regularity on $I$ if and only if there is a constant $C$ such that for all $G\in L^p(I;X)$ one has
\begin{align}\label{eq:SMRgamma}
\Big(\int_I \|s\mapsto \1_{s\leq t}A^{1/2} S(t-s) G(s)\|_{\gamma(I;X)}^p \dd t\Big)^{1/p} \leq C\|G\|_{L^p(I;X)}.
\end{align}
Like deterministic maximal regularity, stochastic maximal $L^p$-regularity also has various permanence properties with respect to $p$ and $I$, see \cite{AV19, LoVer}. 

For $X = L^q$ with $q\in [2, \infty)$, a sufficient condition for stochastic maximal $L^p$-regularity is that $A$ has a bounded $H^\infty$-calculus  of angle $<\pi/2$ (see Theorem \ref{thm:survey} below), which also extends to more general Banach (function) spaces $X$, see \cite{NVW-SMR, NVW-Rbounded}. Currently, no general analogue of the operator-theoretic characterisation is known for stochastic maximal $L^p$-regularity. Stochastic maximal $L^p$-regularity provides an effective tool in the study of nonlinear stochastic evolution equations. For further details and references on stochastic maximal $L^p$-regularity, the reader is referred to \cite{AVsurvey}.

\section{A diagonal counterexample on mixed-norm sequence spaces}
\label{sec:counter}

In this section we construct and analyse the diagonal multiplier operator, which provides the main counterexample of the paper. We first determine its functional calculus properties and give a complete characterisation of its deterministic maximal regularity on the spaces
\[
        X_{q,r}=\ell^q(\ell^r), \qquad q,r \in [1,\infty).
\]
We then combine this characterisation with a concavification argument to obtain the corresponding stochastic maximal regularity result. In the final subsection, we record two further deterministic consequences of the same construction.

\subsection{The diagonal multiplier}
\label{ss:diagonal-multiplier}
Throughout this section we use a diagonal multiplier with exponentially
increasing eigenvalues. For this operator, the functional calculus and
sectoriality properties are immediate, while the mixed-norm structure of
$\ell^q(\ell^r)$ allows us to determine precisely when deterministic and
stochastic maximal regularity hold.

For $q,r\in [1, \infty)$ set
\[
        X_{q,r}:=\ell^q(\ell^r):=\Bigl\{(x_{m,n})_{m,n\ge1}:\norm{x}_{\ell^q(\ell^r)}:=\Bigl(\sum_{m=1}^\infty\Bigl(\sum_{n=1}^\infty|x_{m,n}|^r\Bigr)^{q/r} \Bigr)^{1/q}<\infty \Bigr\}.
\]
We let $e_{j,k}\in X_{q,r}$ denote the sequence such that $(e_{j,k})_{m,n} =1$ if $(m,n) = (j,k)$ and $(e_{j,k})_{m,n} =0$ otherwise. These vectors form a Schauder basis of $X_{q,r}$.

From \cite[Section 4.2]{HNVW1} and \cite[Section 7.1]{HNVW2} we note the following geometric properties of $X_{q,r}$. We refer to \cite{LN24} for the definition of a Banach function space.

\begin{lemma}\label{lem:geometry}
Let $q,r\in [1, \infty)$. The space $X_{q,r}$ is a Banach function space. If $q,r>1$, then $X_{q,r}$ is UMD. If $q,r \geq 2$, then $X_{q,r}$ has type $2$. Moreover, $X_{q,r}$ has cotype $\max\{q,r,2\}$. 
\end{lemma}

Let $A$ be the diagonal operator
\begin{align}\label{eq:Aemndef}
        (Ax)_{m,n}&:=2^n x_{m,n},\qquad
        \Dom(A):=\{x\in X_{q,r}:\ (2^n x_{m,n})_{m,n}\in X_{q,r}\}.
\end{align}
The semigroup generated by $-A$ is given by
\[
        (S(t)x)_{m,n}=e^{-2^nt}x_{m,n}.
\]

\begin{lemma}\label{lem:h-infty}
Let $q,r\in [1, \infty)$. The operator $A$ is $R$-sectorial on $X_{q,r}$ and has a bounded
$H^\infty$-calculus of angle $0$ on this space.
\end{lemma}

\begin{proof}
For $\theta>0$, let
$\lambda_1,\ldots,\lambda_N\notin\overline{\Sigma_\theta}$ and
$x_1,\ldots,x_N\in X_{q,r}$. Since
\[
        \sup_{\lambda\notin\overline{\Sigma_\theta}}
        \sup_{n\geq1}
        \Big|\frac{\lambda}{\lambda-2^n}\Big|
        <\infty,
\]
the Khintchine--Maurey inequalities \cite[Theorem 7.2.13]{HNVW2} give
\[
\begin{aligned}
\E\Big\|\sum_{j=1}^N\varepsilon_j\lambda_jR(\lambda_j,A)x_j\Big\|_{X_{q,r}}^2&\eqsim_{q,r}\Big\|\Big(\sum_{j=1}^N|\lambda_jR(\lambda_j,A)x_j|^2\Big)^{1/2}\Big\|_{X_{q,r}}^2\\
&\lesssim_{q,r,\theta}\Big\|\Big(\sum_{j=1}^N|x_j|^2\Big)^{1/2}\Big\|_{X_{q,r}}^2\eqsim_{q,r}\E\Big\|\sum_{j=1}^N\varepsilon_jx_j\Big\|_{X_{q,r}}^2.
\end{aligned}
\]
Thus $A$ is $R$-sectorial of angle $0$. The bounded $H^\infty$-calculus is immediate: it is given by
$(\varphi(A)x)_{m,n} = \varphi(2^n)x_{m,n}$
and satisfies
\[
        \|\varphi(A)x\|_{X_{q,r}} \leq \|\varphi\|_{\infty}\|x\|_{X_{q,r}}.\qedhere
\]
\end{proof}

\subsection{Deterministic maximal regularity}
\label{ss:diagonal-deterministic}

We first determine exactly when $A$ on $X_{q,r}$ has maximal $L^p$-regularity. For $q>1$, the positive result follows by applying maximal regularity of the one-dimensional diagonal multiplier coordinatewise and using extrapolation. 

\begin{proposition}\label{prop:MRqr}
Let $p,q \in (1, \infty)$ and $r\in [1, \infty)$, or $q=r=1$ and $p\in [1, \infty)$. Let $I= (0,T)$ or $I = \R_+$. Then the operator $A$ on $X_{q,r}$ has maximal $L^p$-regularity on $I$.
\end{proposition}
\begin{proof}
It suffices to consider $I = \R_+$ by \cite[Section 17.2.e]{HNVW3}. Define
\begin{align*}
        (A_rx)_{n}&:=2^n x_{n},\qquad \Dom(A_r):=\{x\in \ell^r:\ (2^n x_{n})_{n}\in \ell^r\},
\end{align*}
for which one can check that $A_r$ has maximal $L^p$-regularity on $\R_+$ for any $p\in (1, \infty)$ as in \cite[Example 17.4.7]{HNVW3}. If $r>1$, the case $p=r$ follows from Fubini's theorem, and extrapolation (see \cite[Theorem 17.2.31]{HNVW3}) gives maximal $L^p$-regularity for all
$p\in(1,\infty)$. Similarly, if $r=1$, maximal $L^1$-regularity  of $A_1$ follows directly from Fubini's theorem, whereas maximal $L^p$-regularity for $p\in(1,\infty)$ of $A_1$ follows by extrapolation again.

Next we consider the operator $A$ on $X_{q,r}$. For $p=q$  take $f \in L^q(\R_+;X_{q,r})$ and let $u_f$ be the mild solution to \eqref{eq:EEMR}. For fixed $m\geq 1$ we have $f_m:=(f_{m,n})_{n\geq 1} \in L^q(\R_+;\ell^r)$. Therefore, by Fubini's theorem and maximal $L^q$-regularity of $A_r$, we have
\[\|A u_f\|_{L^q(\R_+;X_{q,r})}^q = \sum_{m=1}^\infty \|A_r u_{f_m}\|_{L^q(\R_+;\ell^r)}^q \lesssim_{q,r} \sum_{m= 1}^\infty\|f_m\|_{L^q(\R_+;\ell^r)}^q = \|f\|_{L^q(\R_+;X_{q,r})}^q.\]
Hence, $A$ has maximal $L^q$-regularity on $\R_+$. By extrapolation, it also has maximal $L^p$-regularity on $\R_+$ for all $p\in (1, \infty)$. In the case $q=r=1$, the Fubini argument gives maximal
$L^1$-regularity. For $p\in(1,\infty)$, the result follows from extrapolation again.
\end{proof}

The case $q=1$ and $r>1$ behaves differently. Although $A$ on
$X_{1,r}$ is still $R$-sectorial and has a bounded
$H^\infty$-calculus of angle $0$, maximal $L^p$-regularity fails for
every $p\in[1,\infty]$.

\begin{proposition}\label{prop:MRq=1}
Let $p\in [1, \infty]$ and $r\in (1, \infty)$. Let $I= (0,T)$ or $I = \R_+$. Then the operator $A$ on $X_{1,r}$ does not have maximal $L^p$-regularity on $I$.
\end{proposition}
\begin{proof} 
Since $\|S(t)\|\leq e^{-2t}$ by \cite[Theorem 17.2.24]{HNVW3}, it suffices to consider $I=\R_+$. For $f \in L^p(\R_+;X_{1,r})$ define
\[
        Tf(t):=\int_0^t A S(t-s)f(s)\dd s,
\]
or, written out in coordinates,
\[
        (Tf)_{m,n}(t)=\int_0^t 2^n e^{-2^n(t-s)} f_{m,n}(s)\dd s.
\]
Observe that $A$ has maximal $L^p$-regularity on $\R_+$ if and only if $T$ is bounded on $L^p(\R_+;X_{1,r})$.

Let $N\geq 2$ and set $$J^N_m:=((m-1)2^{-N},m 2^{-N}], \qquad m\geq 1.$$ Define $f^N=(f^N_{m,n})_{m,n\geq 1}$ by
\[
        f^N_{m,n}(s):=\1_{\{1,\ldots,2^N\}}(m)\1_{\{1,\ldots,N\}}(n)\1_{J^N_m}(s).
\]
For every $s\in(0,1]$ there is exactly one $m$ such that $s\in J^N_m$, so
\[
        \nrm{f^N(s)}_{X_{1,r}} = \Big(\sum_{n=1}^N 1^r\Big)^{1/r}    =    N^{1/r},        \qquad s\in(0,1].
\]
Therefore
\begin{equation}\label{eq:fNnorm}
    \norm{f^N}_{L^p(\R_+;X_{1,r})}=N^{1/r}.
\end{equation}

We next estimate the norm of $Tf^N$ from below. Fix
$j\in\{2^{N-1}+1,\ldots,2^N\}$ and $t\in J^N_j$,
and let $m\in\{j-2^{N-1},\ldots,j-1\}$. Let
\[
        n_m:=\Big\lfloor \log_2\Big(\frac{2^N}{j-m}\Big)\Big\rfloor .
\]
Then $1\le n_m\le N$ and $\frac12\leq (j-m) 2^{n_m}2^{-N} \le1$. Therefore, noting that $J_m^N\subseteq(0,t)$, we obtain
\[
\begin{aligned}
        (Tf^N)_{m,n_m}(t) &=\int_{J_m^N}2^{n_m} e^{-2^{n_m}(t-s)}\dd s        \\&=e^{-2^{n_m}(t-m2^{-N})}\bigl(1-e^{-2^{n_m}2^{-N}}\bigr)
        \\&\geq e^{-(j-m) 2^{n_m} 2^{-N}}\bigl(1-e^{-2^{n_m}2^{-N}}\bigr) \ge c  2^{n_m}2^{-N}
\end{aligned}
\]
for some absolute constant $c>0$. Thus, we have
\[
\begin{aligned}
        \Big(\sum_{n=1}^N |(Tf^N)_{m,n}(t)|^r\Big)^{1/r}\geq |(Tf^N)_{m,n_m}(t)|\geq c 2^{n_m}2^{-N}\geq\frac{c}{2(j-m)}.
\end{aligned}
\]
Therefore, we obtain
\[
\begin{aligned}
        \norm{Tf^N(t)}_{X_{1,r}}\gtrsim \sum_{m=j-2^{N-1}}^{j-1}\frac1{j-m}        &=\sum_{k=1}^{2^{N-1}}\frac1k \gtrsim N .
\end{aligned}
\]
Since $j\in\{2^{N-1}+1,\ldots,2^N\}$ was arbitrary, this estimate holds for all $$t \in \bigcup_{j=2^{N-1}+1}^{2^N} J_j^N=(\tfrac12,1],$$ and thus
       $ \norm{Tf^N}_{L^p(\R_+;X_{1,r})}
        \gtrsim_p N.$
Since $r>1$, combining this estimate with \eqref{eq:fNnorm}, we conclude that $T$ is not bounded on $L^p(\R_+;X_{1,r})$, and hence $A$ does not have maximal $L^p$-regularity on $\R_+$.
\end{proof}

Combining Propositions \ref{prop:MRqr} and \ref{prop:MRq=1}, we obtain the following full characterisation of deterministic maximal $L^p$-regularity for $A$ on $X_{q,r}$. 

\begin{corollary}\label{cor:finaldet}
Let $p,q,r\in[1,\infty)$, and let $I= (0,T)$ or $I = \R_+$. The operator $A$ on $X_{q,r}$ has maximal $L^p$-regularity on $I$ if and only if one of the following conditions holds:
\begin{enumerate}[label={\rm(\roman*)}, leftmargin=*]
\item\label{it:dmr1} $p>1$ and either $q>1$ or $r=1$;
\item\label{it:dmr2} $p=q=r=1$.
\end{enumerate}
\end{corollary}

\begin{proof}
The sufficiency of the two conditions follows from Proposition
\ref{prop:MRqr}. 
Conversely, suppose that $A$ has maximal $L^p$-regularity. First let $p>1$. If $q>1$, condition \ref{it:dmr1} holds. If $q=1$, Proposition \ref{prop:MRq=1} excludes $r>1$, and hence $r=1$; thus condition \ref{it:dmr1} again holds. 

It remains to consider $p=1$. If $q=1$, Proposition \ref{prop:MRq=1} implies that $r=1$, and condition \ref{it:dmr2} holds. Suppose therefore that $q>1$.
The coordinate projection $P:X_{q,r}\to X_{q,r}$ onto
\[
        Y:=\overline{\operatorname{span}}\{e_{n,n}:n\geq1\}
\]
is contractive and commutes with $A$ and with $S(t)$. Hence maximal $L^1$-regularity of $A$ on $X_{q,r}$ would imply maximal $L^1$-regularity of the restriction $A|_Y$. Under the isometric identification
\[
        Y\to \ell^q,\qquad \sum_{n\geq1}x_ne_{n,n}\mapsto(x_n)_{n\geq1},
\]
the restriction $A|_Y$ corresponds to the diagonal multiplier
$A_qe_n=2^ne_n.$
This contradicts \cite[Example 17.4.7]{HNVW3}.
\end{proof}

For $r\in(1,\infty)$, Lemma \ref{lem:h-infty} and Proposition \ref{prop:MRq=1} show that $A$ on $X_{1,r}$ is $R$-sectorial of angle $0$ and has a bounded $H^\infty$-calculus of angle $0$, but fails maximal $L^p$-regularity for every $p\in[1,\infty]$. Thus the implication from $R$-sectoriality of angle $<\frac{\pi}{2}$ to maximal $L^p$-regularity, which holds on UMD spaces by \cite{Weis-MathAnn}, cannot be extended to arbitrary Banach spaces.
This complements the counterexamples of Kalton--Lancien \cite{KaLa02, KaLa} and Fackler \cite{Fack14,Fackl16}. Their constructions produce generators of bounded analytic semigroups without maximal regularity, with the obstruction arising from the failure of $R$-sectoriality. In Fackler's example \cite{Fackl16}, the semigroup may in addition be chosen positive. By contrast, the diagonal operator
\[
        Ae_{m,n}=2^ne_{m,n}
\]
considered here is positive, $R$-sectorial, and has a bounded
$H^\infty$-calculus of angle $0$, but nevertheless fails maximal
regularity on $\ell^1(\ell^r)$ for $r>1$. This example on $X_{1,r}$ is not stable under duality: its adjoint is not $R$-sectorial. For instance, this follows from Theorem \ref{thm:Xexample} below since $X_{1,2}^*$ does not have finite cotype.  Further deterministic consequences
are discussed in Subsection \ref{ss:further-deterministic}.

\subsection{Stochastic maximal regularity via concavification}
\label{ss:diagonal-stochastic}
For positive multiplication operators on Banach function spaces, stochastic maximal regularity can be reformulated as deterministic maximal regularity on the $2$-concavification. This turns the stochastic problem on $\ell^q(\ell^r)$ into a deterministic problem on $\ell^{q/2}(\ell^{r/2})$. A related equivalence was used in \cite{NVW-SMR} to obtain a counterexample to stochastic maximal $L^2$-regularity on $L^q$ for $q\in(2,\infty)$.

For $q,r\in[2,\infty)$, the $2$-concavification of $X_{q,r}$ is given by
\[
        X_{q,r}^{(2)}:=\ell^{q/2}(\ell^{r/2})
        =X_{q/2,r/2},
\]
and for $x\in X_{q,r}^{(2)}$ we have
\[
        \|x\|_{X_{q,r}^{(2)}}=\big\||x|^{1/2}\big\|_{X_{q,r}}^2.
\]
The equivalence below extends more generally to positive multiplication
operators on $2$-convex UMD Banach function spaces.

\begin{proposition}\label{prop:MRSMR}
Let $p,q,r\in[2,\infty)$, and let $I= (0,T)$ or $I = \R_+$. The following assertions are equivalent:
\begin{enumerate}[label={\rm(\arabic*)}, leftmargin=*]
\item\label{it:MRSMR1}
the operator $A$ on $X_{q,r}$ has stochastic maximal
$L^p$-regularity on $I$;
\item\label{it:MRSMR2}
the operator $A$ on $X_{q/2,r/2}$ has maximal
$L^{p/2}$-regularity on $I$.
\end{enumerate}
\end{proposition}

\begin{proof}
Although the proof is a straightforward adaptation of \cite{NVW-SMR}, we indicate the main steps. Put $Y:=X_{q/2,r/2}$
and define
\[
        (\mc Tf)(t)    :=    \int_0^t AS(2(t-s))f(s)\dd s.
\]
For every strongly progressively measurable simple process $G$, the
square-function estimate \eqref{eq:Ito2} gives
\begin{align*}
            \|A^{1/2}U_G\|_{L^p(\Omega\times I;X_{q,r})}&\eqsim_{p,q,r} \brB{\int_I \E\Big\| \int_0^t |A^{1/2}S(t-s)G(s)|^2\dd s  \Big\|_Y^{p/2}\dd t}^{1/p}\\&=\|\mc T(|G|^2)\|_{L^{p/2}(\Omega\times I;Y)}^{1/2}.
\end{align*}

Suppose first that the operator $A$ on $Y$ has maximal $L^{p/2}$-regularity on $I$. Then so does $2A$, so $\mc T$ is bounded on $L^{p/2}(I;Y)$. Applying this estimate pathwise to $|G|^2$ proves stochastic maximal $L^p$-regularity of $A$ on $X_{q,r}$.

Conversely, suppose that $A$ on $X_{q,r}$ has stochastic maximal $L^p$-regularity. Let $f\geq0$ be a simple $Y$-valued function and set $G=f^{1/2}$. The preceding identities and stochastic maximal regularity imply
\[
\|\mc{T}f\|_{L^{p/2}(I;Y)}\lesssim\|G\|_{L^p(I;X_{q,r})}^2=\|f\|_{L^{p/2}(I;Y)}.
\]
The same estimate  for general  $f\in L^{p/2}(I;X_{\frac{q}{2},\frac{r}{2}})$ follows by
splitting $f$ into its (real and imaginary) positive and negative parts and using density.
By the definition of $\mc{T}$ this implies that $2A$, and thus $A$, has maximal $L^{p/2}$-regularity on $I$.
\end{proof}

Combining Proposition \ref{prop:MRSMR} with Corollary
\ref{cor:finaldet} yields a full classification of stochastic maximal $L^p$-regularity for $A$ on $X_{q,r}$. 

\begin{theorem}\label{thm:mainSMR}
Let $p,q,r\in[2,\infty)$, and let the operator $A$ on $X_{q,r}$ be defined by
\eqref{eq:Aemndef}. Let $I= (0,T)$ or $I = \R_+$. Then $A$ has stochastic maximal
$L^p$-regularity on $I$ if and only if one of the following conditions
holds:
\begin{enumerate}[label={\rm(\roman*)}, leftmargin=*]
\item $p>2$ and either $q>2$ or $r=2$;
\item $p=q=r=2$.
\end{enumerate}
\end{theorem}

\begin{proof}
By Proposition \ref{prop:MRSMR}, stochastic maximal $L^p$-regularity of $A$ on $X_{q,r}$ is equivalent to maximal $L^{p/2}$-regularity of $A$ on $X_{q/2,r/2}$. The result therefore
follows directly from Corollary \ref{cor:finaldet}.
\end{proof}

In particular, the critical choice $q=2$ and $r>2$ produces the following counterexample.

\begin{corollary}\label{cor:main}
There exist a UMD Banach function space $X$ of type $2$ and a sectorial operator $A$ on $X$ with a bounded $H^\infty$-calculus of angle $0$ such that $A$ fails stochastic maximal $L^p$-regularity for every $p\in[2,\infty)$.
\end{corollary}

\begin{proof}
Fix $r\in(2,\infty)$, take $X=X_{2,r}$, and let $A$ be defined by \eqref{eq:Aemndef}. By Lemma \ref{lem:geometry}, $X$ is a UMD Banach function space of type $2$, and by Lemma \ref{lem:h-infty}, $A$ has a bounded $H^\infty$-calculus of angle $0$. The failure of stochastic maximal $L^p$-regularity for every $p\in[2,\infty)$ follows from Theorem \ref{thm:mainSMR}.
\end{proof}

The mixed-norm spaces occurring in this counterexample are  natural from the point of view of applications. Indeed, Besov spaces $B^s_{r,q}$ admit sequence-space representations modelled on $\ell^q(\ell^r)$, and such spaces arise naturally in applications of deterministic and stochastic maximal regularity to partial differential equations. In these applications the endpoint $q=1$ can be important; see \cite{agresti2023primitive,DHMPT,OgSh}.

\subsection{An extension of the deterministic construction}
\label{ss:further-deterministic}

We conclude the section with  an extension of our construction which yields examples in which both the operator and its adjoint are $R$-sectorial. This result isolates the two geometric
properties governing the construction: finite cotype is equivalent to
$R$-sectoriality, whereas the UMD property is equivalent to maximal
regularity.

Let $X$ be an order-continuous Banach function space and define
\[
        X(\ell^2):=\cbraceB{(x_n)_{n\geq1}\subseteq X:\nrmB{\brB{\sum_{n=1}^\infty |x_n|^2}^{1/2}}_X<\infty}.
\]
Then $X(\ell^2)$ has order-continuous norm. In particular, the
finitely nonzero sequences are dense in $X(\ell^2)$. Moreover $X(\ell^2)$ has finite cotype if and only if $X$ has finite cotype, and $X(\ell^2)$ is UMD if and only if $X$ is UMD.

Define the diagonal operator $A$ on $X(\ell^2)$ by
\begin{align}\label{eq:def-A}
        (Ax)_n:=2^nx_n,
        \qquad
        \Dom(A):=
        \left\{
        (x_n)_{n\geq1}\in X(\ell^2):
        (2^nx_n)_{n\geq1}\in X(\ell^2)
        \right\}.
\end{align}
The operator $A$ in
\eqref{eq:def-A} is densely defined, and $-A$ generates the analytic
$C_0$-semigroup given by
\[
        (S(t)x)_n=e^{-2^nt}x_n.
\]

\begin{theorem}\label{thm:Xexample}
Let $X$ be an order-continuous Banach function space, let $p\in(1,\infty)$, and let $I= (0,T)$ or $I = \R_+$. The operator $A$ defined by \eqref{eq:def-A} is sectorial on $X(\ell^2)$ and has a bounded $H^\infty$-calculus of angle $0$ on this space. Furthermore, the following assertions hold:
\begin{enumerate}[label={\rm(\arabic*)}, leftmargin=*]
  \item\label{it:1Xexample} $A$ is $R$-sectorial of angle $0$ if and only if $X$ has finite cotype;
  \item\label{it:2Xexample} $A$ has maximal $L^p$-regularity on $I$ if and only if $X$ has UMD.
\end{enumerate}
\end{theorem}

\begin{proof}
The sectoriality and boundedness of the $H^\infty$-calculus of angle $0$ follow as in Lemma \ref{lem:h-infty}.
The proofs of \ref{it:1Xexample} and \ref{it:2Xexample} are carried out in four steps.

\smallskip
{\em Step 1: Finite cotype implies $R$-sectoriality.} 
Suppose that $X$ has finite cotype. Then $X$ is $q$-concave for some $q\in [2, \infty)$ by \cite[Corollary 1.f.9]{LT79}. By Minkowski's inequality $X(\ell^2)$ is $q$-concave as well. From \cite[Proposition 1.f.3]{LT79} it follows that $X(\ell^2)$ has finite cotype, and thus has the triangular contraction property by \cite[Theorem 7.5.20]{HNVW2}. Therefore, the $R$-sectoriality follows from \cite[Theorem 10.3.4(2)]{HNVW2}.

\smallskip
{\em Step 2: $R$-sectoriality implies finite cotype.}
Suppose that $A$ is $R$-sectorial on $X(\ell^2)$ and assume that $X$ does not have finite cotype. Fix $N\geq 1$ and let $e_1,\ldots,e_M$ be a $1/2$-net in the unit ball of $\ell_N^2$, i.e.,
$$
\cbraceb{e \in \ell^2_N:\nrm{e}_{\ell^2_N} \leq 1} \subseteq \bigcup_{m=1}^M \cbraceb{e \in \ell^2_N:\nrm{e-e_m}_{\ell^2_N} \leq \tfrac12}.
$$
By the lattice Maurey--Pisier theorem \cite[Theorem 1.f.12]{LT79}, there are pairwise
disjointly supported positive functions $x_1,\ldots,x_M\in X$ such that
\begin{equation}\label{eq:MP}
        \frac12\max_{1\le m\le M}|c_m| \le \nrmB{\sum_{m=1}^M c_mx_m}_X\le \max_{1\le m\le M}|c_m|
\end{equation}
for all scalars $c_1,\ldots,c_M$.

For $j=1,\ldots,N$ put
\[
        T_j:=2^j(2^j+A)^{-1}=(-2^j)R(-2^j,A).
\]
Then, for $x=(x_n)_{n\ge1}\in X(\ell^2)$,
\[
        T_jx = \Bigl(\frac{x_n}{1+2^{n-j}}\Bigr)_{n\ge1} =\bigl(a_{n-j}x_n\bigr)_{n\ge1},    \qquad    a_k:=\frac{1}{1+2^k}.
\]
By assumption the family $\{T_j:j\ge1\}$ is $R$-bounded on
$X(\ell^2)$.

Let $B=(b_{j,n})_{j,n=1}^N$ be a scalar matrix. For $j=1,\ldots,N$, define
$y^j\in X(\ell^2)$ by
\[
        (y^j)_n:=\sum_{m=1}^M (e_m)_j b_{j,n}x_m,    \qquad 1\le n\le N,
\]
and put $(y^j)_n=0$ for $n>N$. Using \eqref{eq:MP} and the disjointness of the $x_m$'s, for every choice of signs
$(\varepsilon_j)_{j=1}^N$ we have
\begin{align*}
 \nrmB{\sum_{j=1}^N \varepsilon_j y^j
        }_{X(\ell^2)} \eqsim\max_{1\le m\le M}\brB{\sum_{n=1}^N\absB{\sum_{j=1}^N \varepsilon_j(e_m)_j b_{j,n}}^2}^{1/2}.
\end{align*}
Therefore, since the $e_m$'s form a $1/2$-net of the unit ball of $\ell^2_N$, we have
\begin{align*}
  \|B\|_{\ms{L}(\ell_N^2)} = \left\|(\varepsilon_j b_{j,n})_{j,n=1}^N\right\|_{\ms L(\ell_N^2)} &\eqsim \max_{1\le m\le M} \brB{\sum_{n=1}^N \absB{\sum_{j=1}^N (e_m)_j \varepsilon_j b_{j,n} }^2}^{1/2} \\&\eqsim \nrmB{\sum_{j=1}^N \varepsilon_j y^j  }_{X(\ell^2)}.
\end{align*}
The same computation applied to $T_jy^j$ yields
\[
\|(a_{n-j}b_{j,n})_{j,n=1}^N\|_{\mc{L}(\ell_N^2)} \eqsim  \nrmB{\sum_{j=1}^N \varepsilon_j T_j y^j}_{X(\ell^2)}.
\]
Letting the $\varepsilon_j$'s be the pointwise evaluations of a Rademacher sequence, taking expectations in the preceding two displays and using the $R$-boundedness of $\{T_j:j\ge1\}$, we obtain
\begin{equation}\label{eq:smoothtoeplitz}
  \left\| (a_{n-j}b_{j,n})_{j,n=1}^N \right\|_{\mathscr L(\ell_N^2)}  \lesssim\|B\|_{\mathscr L(\ell_N^2)}
\end{equation}
with a constant only depending on the $R$-sectoriality constant of $A$.

It remains to replace the smooth Toeplitz--Schur multiplier by the sharp lower triangular Schur projection. Note that
\[
        \mathbf \1_{\{k\le0\}}-a_k =   \mathbf \1_{\{k\le0\}}\frac{2^k}{1+2^k} - \1_{\{k>0\}}\frac{1}{1+2^k}   \in\ell^1(\mathbb Z).
\]
Since Toeplitz--Schur multipliers on $\mc L(\ell_N^2)$ whose kernels belong to $\ell^1(\Z)$ are bounded uniformly in $N$, \eqref{eq:smoothtoeplitz} implies that the family of lower triangular Schur projections
\[
        \Delta_NB := \bigl(\mathbf \1_{\{n-j\le 0\}}b_{j,n}\bigr)_{j,n=1}^N
\]
is uniformly bounded on $\mathscr L(\ell_N^2)$. This contradicts the Kwapie\'n--Pe{\l}czy\'nski triangular-projection theorem, see \cite[Lemma 7.5.12]{HNVW2}. Thus $X$ has finite cotype, finishing the proof of \ref{it:1Xexample}.
\smallskip

As  in Proposition \ref{prop:MRq=1}, in the proof of \ref{it:2Xexample} it suffices to work with $I=\R_+$.

\smallskip
{\em Step 3: The UMD property implies maximal regularity.}
Assume that $X$, and thus $X(\ell^2)$, is UMD. By \ref{it:1Xexample}, the operator $A$ is $R$-sectorial of angle $0$ on $X(\ell^2)$, which implies that $A$ has maximal $L^p$-regularity on $\R_+$.

\smallskip
{\em Step 4: Maximal regularity implies the UMD property.}
Assume that $A$ has maximal $L^p$-regularity on $\R_+$. Since $0\in\varrho(A)$, \cite[Theorem 17.3.31]{HNVW3} yields that the convolution operator
\[
        Tf(t)=\int_{-\infty}^t Ae^{-(t-s)A}f(s)\dd s,
        \qquad t\in\R,
\]
is bounded on $L^p(\R;X(\ell^2))$. Moreover, by \cite[Theorem 17.3.1]{HNVW3} we know that $A$ is $R$-sectorial, and thus that $X$ has finite cotype by \ref{it:1Xexample}.

 Let $D_t$ be the differentiation operator on $L^p(\mathbb R;X)$ with domain $W^{1,p}(\mathbb R;X)$. Take $f = (f_n)_{n\geq 1}  \in L^p(\mathbb R;X(\ell^2))$. Then note that for $t \in \R$ we have
\begin{align*}
  (Tf)_n(t) &=\int_{-\infty}^t 2^ne^{-2^n(t-s)}f_n(s)\dd s= \bigl(2^n(2^n+D_t)^{-1}f_n\bigr)(t)
\end{align*}
and therefore the boundedness of $T$ on $L^p(\mathbb{R};X(\ell^2))$ implies that
\begin{align*}
  \cbrace{2^n(2^n+D_t)^{-1}:n\geq 1}
\end{align*}
is $\ell^2$-bounded.

For $\lambda>0$, the operator
$\lambda(\lambda+D_t)^{-1}$ is convolution with the positive kernel
\[
        k_\lambda(s) := \lambda e^{-\lambda s}\1_{(0,\infty)}(s).
\]
If $2^n\leq\lambda<2^{n+1}$, then
\[
        0\leq k_\lambda(s)\leq 2k_{2^n}(s),  \qquad s\in\R.
\]
By positivity, the $\ell^2$-boundedness of $\{2^n(2^n+D_t)^{-1}:n\geq1\}$ therefore implies the $\ell^2$-boundedness of $\{\lambda(\lambda+D_t)^{-1}:\lambda\geq2\}.$ Dilation invariance of $D_t$ then gives the same conclusion for all $\lambda>0$. Therefore, by (the proof of) \cite[Theorem 2.4.9]{KLW19} we conclude that $X$ is UMD.
\end{proof}

Using Theorem \ref{thm:Xexample} with a non-UMD Banach function space $X$ of non-trivial type gives the following consequence. Such spaces exist by \cite[Section 4.3.c]{HNVW1}. Moreover, $X$ and $X^*$ have finite cotype by \cite[Proposition 7.1.13 and Theorem 7.1.14]{HNVW2}.

\begin{corollary}\label{cor:dual-stable-counterexample}
There exist a Banach function space $X$ and a sectorial operator $A$ on $X$ such that both $A$ and $A^*$ are $R$-sectorial and have a bounded $H^\infty$-calculus of angle $0$, but neither $A$ nor $A^*$ has maximal $L^p$-regularity on any bounded interval or on $\R_+$, for any $p\in(1,\infty)$.
\end{corollary}
\begin{proof} 
Choose a non-UMD Banach function space $X_0$ with non-trivial type, as in \cite[Section 4.3.c]{HNVW1}. Then $X_0$ and $X_0^*$ have finite cotype. In particular, both have order-continuous norm, and the canonical duality identifies $\bigl(X_0(\ell^2)\bigr)^* =  X_0^*(\ell^2).$ Put $X:=X_0(\ell^2)$ and let $A$ be the diagonal operator defined by \eqref{eq:def-A}. Under the preceding duality, $A^*$ is the corresponding diagonal operator on $X_0^*(\ell^2)$.

Since both $X_0$ and $X_0^*$ have finite cotype, Theorem \ref{thm:Xexample} shows that $A$ and $A^*$ are $R$-sectorial of angle $0$. Both operators have a bounded $H^\infty$-calculus of angle $0$. The UMD property is self-dual. Since $X_0$ is not UMD, neither is $X_0^*$. Another application of Theorem \ref{thm:Xexample} shows that neither $A$ nor $A^*$ has maximal $L^p$-regularity, for any $p\in(1,\infty)$, on any bounded interval or on $\R_+$.
\end{proof}

To the best of our knowledge, it remains open whether an example similar to that in Corollary \ref{cor:dual-stable-counterexample}, without the $H^\infty$-calculus requirement, can be constructed on $\ell^1$ or $L^1$. Standard ways of transferring the preceding construction to these spaces appear to destroy $R$-sectoriality. 

On the other hand, on $\ell^1$ and $L^1$, a bounded $H^\infty$-calculus of angle $<\frac{\pi}{2}$ implies maximal $L^p$-regularity for every $p\in[1,\infty)$. For $p>1$ this is \cite[Theorem 7.5]{KWcalc}. We briefly indicate why the endpoint $p=1$ follows from a similar argument. Since $\ell^1$ and $L^1$ are GT-spaces of cotype $2$, for
$\omega_{H^\infty}(A)<\nu<\frac{\pi}{2}$ and $0<s<1$ it follows from \cite[Proposition 7.1]{KWcalc} that
\[
\int_{\Gamma_\nu} \|A^sR(\zeta,A)x\|\frac{|\dd\zeta|}{|\zeta|^s}\lesssim \|x\|, \qquad x\in X.
\]
Using the contour representation from \cite[Proposition 4.2]{KWcalc},
\[
        Ae^{-tA}x =-\frac{1}{2\pi i} \int_{\Gamma_\nu} \zeta^{1-s}e^{-t\zeta}A^sR(\zeta,A)x\dd\zeta,
\]
and the estimate $\operatorname{Re}\zeta\gtrsim|\zeta|$ on $\Gamma_\nu$, integration first with respect to $t$ gives
\[
        \int_0^\infty\|Ae^{-tA}x\|\dd t \lesssim \int_{\Gamma_\nu} \|A^sR(\zeta,A)x\|\frac{|\dd\zeta|}{|\zeta|^s} \lesssim \|x\|.
\]
It follows from Fubini's theorem that the operator
\[
        Tf(t) := \int_0^t Ae^{-(t-s)A}f(s)\dd s
\]
is bounded on $L^1(\R_+;X)$. For simple functions with values in $\Dom(A)$, the corresponding mild solution is a strong solution and $Au_f=Tf$. By density and closedness, the same conclusion holds for all $f\in L^1(\R_+;X)$. Thus $A$ has maximal $L^1$-regularity. Note that the case $p>1$ follows by extrapolation. 

\section{On the \texorpdfstring{$R$}{R}-boundedness of stochastic convolution operators}\label{sec:Rbdd}
The counterexample in Theorem \ref{thm:mainSMR} shows that a bounded $H^\infty$-calculus alone does not imply stochastic maximal regularity on UMD spaces with type $2$. In this section, we study the additional geometric condition underlying the known positive results. This condition is formulated as an $R$-boundedness property for elementary stochastic convolution operators. We consider two versions: condition $(S_p)$ involving normalised interval kernels used in \cite{NVW-SMR,NVW-Rbounded}, and condition $(S_p^{\exp})$ involving exponential kernels, which was introduced in \cite{NVW-survey}. 
The respective proofs of stochastic maximal regularity in these works rely on these two apparently different formulations.

In Subsection \ref{ss:triangular}, we prove that $(S_p)$ and $(S_p^{\exp})$ are equivalent. 
We then show in Subsection \ref{ss:casep2} that the endpoint condition $(S_2)$ holds precisely for those Banach spaces that are isomorphic to a Hilbert space.
Finally, we show that condition $(S_p)$ is not only a sufficient criterion for stochastic maximal regularity, but is also necessary for stochastic maximal regularity of two canonical operators with bounded $H^\infty$-calculi of angle $0$. Indeed, Subsection \ref{ss:Schauder} treats a diagonal Schauder multiplier on $\Rad^p(X)$, whereas Subsection \ref{ss:Laplace} treats the Laplacian $-\Delta$ on $L^q(\R^d;X)$. In both cases, stochastic maximal $L^p$-regularity is characterised by condition $(S_p)$ on the underlying space $X$.

\subsection{Definition and properties}\label{ss:RbddX}
For unexplained terminology, the reader is referred to \cite{HNVW1, HNVW2, HNVW3}. Let $X$ be a UMD Banach space with type $2$ and let $p\in[2,\infty)$. Let $k\in L^2(\R_+)$ and let $S_k$ denote the stochastic convolution operator
\[
    (S_k G)(t) :=\int_0^t k(t-s) G(s)\dd W(s), \qquad t\geq 0,
\]
where $G\in L^p_{\mathscr F}(\Omega\times \R_+;X)$. Here the subscript $\mathscr{F}$ stands for the subspace of progressively measurable processes.
Since $X$ has type $2$, the It\^o isomorphism and Young's
inequality give
\[
        \|S_kG\|_{L^p(\Omega\times\R_+;X)}\lesssim_{p,X}\|k\|_{L^2(\R_+)}        \|G\|_{L^p(\Omega\times\R_+;X)}.
\]
Thus $S_k$ extends uniquely to
$L^p_{\mathscr F}(\Omega\times\R_+;X)$.

For $\lambda\in\C$ with $\Re\lambda>0$, define
\[
        k_\lambda(t):=\lambda^{1/2}e^{-\lambda t},
        \qquad t>0,
\]
where the principal branch of the square root is used.
For $\delta>0$, define
\[
        h_\delta(t):=\delta^{-1/2}\1_{(0,\delta)}(t),
        \qquad t>0.
\]
\begin{definition}\label{def:Rbddconv}
Let $p \in [2,\infty)$ and let $X$ be a UMD space with type $2$.
\begin{enumerate}[\rm (1)]
\item The space $X$ is said to satisfy condition $(S_p^{\exp})$ if for every $\nu\in (0,\pi/2)$ the family $\{S_{k_\lambda}:\lambda\in \Sigma_{\nu}\}$
is $R$-bounded as a subset of
\[
    \mathscr L\big(L^p_{\mathscr F}(\Omega\times\mathbb R_+;X),L^p(\Omega\times\mathbb R_+;X) \big).
\]
\item The space $X$ is said to satisfy condition $(S_p)$ if the family $\{S_{h_{\delta}}:\delta>0\}$
is $R$-bounded as a subset of $$ \mathscr L\big(L^p_{\mathscr F}(\Omega\times\mathbb R_+;X),L^p(\Omega\times\mathbb R_+;X) \big).$$
\end{enumerate}
\end{definition}

One can check that the $R$-boundedness of $\{S_{k_\lambda}:\lambda>0\}$ immediately implies $(S_p^{\exp})$. Indeed, for $\lambda_n\in \Sigma_{\nu}$ this follows by writing
\begin{align*}
\varepsilon_n (S_{k_{\lambda_n}} G_n)(t)  =  \varepsilon_n \frac{\lambda^{1/2}_n}{\Re(\lambda_n)^{1/2}} e^{-i\Im(\lambda_n)t}  \int_0^t  k_{\Re(\lambda_n)}(t-s) e^{i\Im(\lambda_n)s} G_n(s)\dd W(s),
\end{align*}
and using the  Kahane contraction principle with $\frac{|\lambda|}{|\Re(\lambda)|} \leq C_{\nu}$. The converse will be proved in Subsection \ref{ss:triangular}.  

Given $k\in L^2(\R_+)$, let $N_k:L^p(\R_+;X)\to L^p(\R_+;\gamma(\R_+;X))$ be defined by
\[N_kG = t \mapsto [s\mapsto \1_{0<s<t}k(t-s)G(s)].\]
The following lemma is a consequence of the It\^o isomorphism in UMD spaces and can be proved in the same way as \cite[Theorem 7.1]{NVW-Rbounded}.
\begin{lemma}\label{lem:gammaRbdd}
Let $X$ be a UMD space with type $2$, let
$p\in[2,\infty)$, and let $I=(0,T)$ or $I=\R_+$. Let $\mathcal{K}\subseteq L^2(\R_+)$.
Then the following assertions are equivalent:
\begin{enumerate}[label={\rm(\arabic*)}, leftmargin=*]
\item The family $\{S_k:k\in \mathcal{K}\}$ is $R$-bounded on $L^p_{\mathscr F}(\Omega\times I;X)$.
\item The family $\{N_k:k\in \mathcal{K}\}$ is $R$-bounded from $L^p(I;X)$ into $L^p(I;\gamma(I;X))$.
\end{enumerate}
\end{lemma}

The definitions above are made on $\mathbb R_+$, but in applications one
often works on finite intervals. The following elementary observation shows
that, for dilation-invariant families of kernels, this distinction is
immaterial. Thus the $R$-boundedness conditions may be checked on any
convenient finite time interval.

\begin{lemma}[Independence of time interval]\label{lemma:timeintervalSp}
Let $X$ be a UMD space with type $2$, let
$p\in[2,\infty)$, and let $T\in (0,\infty)$. Suppose that $\mathcal{K}\subseteq L^2(\R_+)$ is invariant under dilation, i.e., $k\in \mathcal{K} \Rightarrow t\mapsto \theta^{1/2} k(\theta t)\in \mathcal{K}$ for all $\theta>0$. Then the following assertions are equivalent:
\begin{enumerate}[label={\rm(\arabic*)}, leftmargin=*]
\item The family $\{S_k:k\in \mathcal{K}\}$ is $R$-bounded on $L^p_{\mathscr F}(\Omega\times\mathbb R_+;X)$.
\item The family $\{S_k:k\in \mathcal{K}\}$ is $R$-bounded on $L^p_{\mathscr F}(\Omega\times(0,T);X)$.
\end{enumerate}
\end{lemma}
\begin{proof}
The implication from $\R_+$ to $(0,T)$ follows by restriction. We prove the converse, for which it suffices to consider $N_k$ by Lemma \ref{lem:gammaRbdd}.  Write $N_k^T$ when the time interval is $(0,T)$. Suppose that the family
$\{N_k^T:k\in\mathcal K\}$
is $R$-bounded, with $R$-bound at most $C$.
Fix $T'>0$ and put $\theta:=\frac{T'}{T}.$ For $k\in\mathcal K$, define
\[
        k^{(\theta)}(u):=\theta^{1/2}k(\theta u), \qquad u>0.
\]
By dilation invariance, $k^{(\theta)}\in\mathcal K$. For
$G\in L^p(0,T';X)$, put
\[
        G_\theta(r):=G(\theta r),
        \qquad 0<r<T,
\]
and let $V_\theta:L^2(0,T')\longrightarrow L^2(0,T)$ be given by
\[
        (V_\theta f)(r):=\theta^{1/2}f(\theta r).
\]
Then $V_\theta$ is unitary. A change of variables gives
\[
        (N_k^{T'}G)(\theta\tau)
        =
        (N_{k^{(\theta)}}^TG_\theta)(\tau)\circ V_\theta,
        \qquad 0<\tau<T.
\]
The ideal property of $\gamma$-radonifying operators therefore gives
\[
        \|(N_k^{T'}G)(\theta\tau)\|_{\gamma(0,T';X)}
        =
        \|(N_{k^{(\theta)}}^TG_\theta)(\tau)\|_{\gamma(0,T;X)}.
\]

Let $k_1,\ldots,k_N\in\mathcal K$ and
$G_1,\ldots,G_N\in L^p(0,T';X)$. It follows that
\[
\begin{aligned}
&\Big(
\E_\varepsilon
\Big\|
\sum_{j=1}^N\varepsilon_jN_{k_j}^{T'}G_j
\Big\|_{L^p(0,T';\gamma(0,T';X))}^2
\Big)^{1/2}
\\
&\qquad
=
\theta^{1/p}
\Big(
\E_\varepsilon
\Big\|
\sum_{j=1}^N
\varepsilon_jN_{k_j^{(\theta)}}^TG_{j,\theta}
\Big\|_{L^p(0,T;\gamma(0,T;X))}^2
\Big)^{1/2}
\\
&\qquad
\leq
C\theta^{1/p}
\Big(
\E_\varepsilon
\Big\|
\sum_{j=1}^N\varepsilon_jG_{j,\theta}
\Big\|_{L^p(0,T;X)}^2
\Big)^{1/2}
\\
&\qquad
=
C
\Big(
\E_\varepsilon
\Big\|
\sum_{j=1}^N\varepsilon_jG_j
\Big\|_{L^p(0,T';X)}^2
\Big)^{1/2}.
\end{aligned}
\]
Thus the $R$-bound on $(0,T')$ is independent of $T'$. Applying this estimate to restrictions to $(0,T')$ and letting $T'\to\infty$, monotone convergence gives the asserted $R$-boundedness on $\R_+$. 
\end{proof}

The following result was proved for $X = L^q$ with $q\in [2, \infty)$  under the condition $(S_p)$ in \cite{NVW-SMR} and for general $X$ under the condition $(S_p^{\exp})$ in \cite{NVW-survey}. 
The papers \cite{NVW-SMR,NVW-survey} provide our main motivations to consider the $R$-boundedness conditions of Definition~\ref{def:Rbddconv}.  
\begin{theorem}\label{thm:survey}
Let $X$ be a UMD Banach space with type $2$. Let $p\in [2, \infty)$ and suppose that $(S_p^{\exp})$ holds. If $A$ has a bounded $H^\infty$-calculus of angle $<\pi/2$, then $A$ has stochastic maximal $L^p$-regularity.
\end{theorem}

By a convexity argument one can show that the condition $(S_p)$ implies that a large family of stochastic convolutions becomes $R$-bounded (see \cite[Proposition 3.2]{NVW-SMR}). In particular, this includes the exponential ones, and therefore $(S_p)$ implies $(S_p^{\exp})$. In the next subsection we will prove the converse result.

\subsection{Equivalence of \texorpdfstring{$(S_p)$ and $(S_p^{\exp})$}{Sp and Spexp} and the triangular contraction property}\label{ss:triangular}

To compare interval kernels with exponential kernels, we need to multiply Rademacher sums by scalar coefficient matrices depending on two indices, for which we introduce some terminology. The space $X$ is said to have the {\em triangular contraction property} if there is a constant $C_\Delta$ such that for all finite families
$(x_{ij})_{i,j=1}^n\subseteq X$,
\[
    \Big(\E \wt{\E}\Big\| \sum_{1\leq i\leq j\leq n}\varepsilon_i\wt{\varepsilon}_j x_{ij}\Big\|_X^2\Big)^{1/2}\leq C_\Delta \Big(\E\wt{\E}\Big\|           \sum_{i,j=1}^n\varepsilon_i\wt{\varepsilon}_j x_{ij}\Big\|_X^2 \Big)^{1/2}, 
\]
and {\em Pisier's contraction property} if there is a constant $C$ such that for all scalars $(\alpha_{ij})_{i,j=1}^n$ and $(x_{ij})_{i,j=1}^n\subseteq X$
\[
    \Big(\E\wt{\E}\Big\|\sum_{i,j=1}^n \alpha_{ij} \varepsilon_i\wt{\varepsilon}_j x_{ij} \Big\|_X^2\Big)^{1/2}\leq C \sup_{1\leq i,j\leq n} |\alpha_{ij}|\Big(\E\wt{\E}\Big\|            \sum_{i,j=1}^n\varepsilon_i\wt{\varepsilon}_j x_{ij}\Big\|_X^2\Big)^{1/2}.
\]
In the literature this is also referred to as property $(\alpha)$. Unfortunately not all UMD spaces have Pisier's contraction property. For instance, the Schatten class operators $\mathcal{C}^p$ for $p\in (1, \infty)\setminus \{2\}$ give a counterexample (see \cite[Proposition 7.5.6]{HNVW2}). On the other hand, all UMD spaces have the triangular contraction property (see \cite[Theorem 5.6.9]{HNVW2}). In Banach function spaces Pisier's contraction property and the triangular contraction property are equivalent to finite cotype (see \cite[Theorem 7.5.20]{HNVW2}).

The aim of this section is to provide conditions on the coefficients $(\alpha_{ij})$ under which the weaker triangular contraction property suffices. It is well-known that certain Fourier multiplier operators emerging from multipliers of bounded variation form an $R$-bounded set \cite[Theorem 8.3.4]{HNVW2}. The following can be seen as an analogous result for pointwise multiplication on $\Rad^2(X)$.

\begin{lemma}[Bounded variation coefficients and triangular contractions]
\label{lem:BV-triangular-contraction}
Let $X$ be a Banach space with the triangular contraction property.
Let $M,N\geq1$ and let $(\alpha_{ij})_{i,j= 1}^{M,N}$ be a
scalar matrix. Suppose that its columns have uniformly bounded variation, i.e.
\[
    B:=\sup_{1\leq j\leq N}\Bigl(|\alpha_{Mj}|+\sum_{i=1}^{M-1}|\alpha_{ij}-\alpha_{i+1,j}|\Bigr) <\infty .
\]
Then for all $(x_{ij})_{i,j=1}^{M,N}\subseteq X$,
\[
    \Big(\E\wt{\E} \Big\| \sum_{i=1}^M\sum_{j=1}^N          \varepsilon_i\wt{\varepsilon}_j \alpha_{ij}x_{ij} \Big\|_X^2 \Big)^{1/2} \lesssim    C_\Delta B \Big( \E\wt{\E} \Big\|\sum_{i=1}^M\sum_{j=1}^N \varepsilon_i\wt{\varepsilon}_j x_{ij} \Big\|_X^2
    \Big)^{1/2}.
\]
The implicit constant is independent of $M,N$, the coefficients
 $(\alpha_{ij})_{i,j= 1}^{M,N}$, and the vectors $(x_{ij})_{i,j=1}^{M,N}$.
\end{lemma}

\begin{proof}
We first claim that for  $k_1,\ldots,k_N\in\{0,1,\ldots,M\}$ we have 
\begin{equation}\label{eq:firstclaimtriangular}
       \Big(\E\wt{\E}\Big\|\sum_{i=1}^M\sum_{j=1}^N\varepsilon_i\wt{\varepsilon}_j\mathbf \1_{\{i\leq k_j\}}x_{ij}\Big\|_X^2\Big)^{1/2}
    \leq C_\Delta\Big(\E\wt{\E}\Big\| \sum_{i=1}^M\sum_{j=1}^N\varepsilon_i\wt{\varepsilon}_jx_{ij}\Big\|_X^2\Big)^{1/2}. 
\end{equation}
Indeed, introduce an independent Rademacher sequence
$(\eta_k)_{k=0}^M$. Conditionally on $(\eta_k)_{k=0}^M$, the random variables $(\eta_{k_j}\wt{\varepsilon}_j)_{j=1}^N$ have the same distribution as $(\wt{\varepsilon}_j)_{j=1}^N$. Therefore, defining
$y_{ik}:=\sum_{\{j:\ k_j=k\}}\wt{\varepsilon}_jx_{ij}$ for $1\leq i\leq M,$ $0\leq k\leq M$ and applying the triangular contraction property, we have
\begin{align*}
\Big(\E\wt{\E}\Big\|\sum_{i=1}^M\sum_{j=1}^N\varepsilon_i\wt{\varepsilon}_j\mathbf \1_{\{i\leq k_j\}}x_{ij}\Big\|_X^2\Big)^{1/2} &=\Big(\E\mathbb E_{\eta}\wt{\E}\Big\|\sum_{i=1}^M\sum_{j=1}^N           \varepsilon_i\eta_{k_j}\wt{\varepsilon}_j\mathbf \1_{\{i\leq k_j\}}x_{ij} \Big\|_X^2\Big)^{1/2}\\&= \Big(\wt{\E}\E\mathbb E_{\eta}\Big\|\sum_{i=1}^M\sum_{k=0}^M\varepsilon_i\eta_k\mathbf \1_{\{i\leq k\}}y_{ik}\Big\|_X^2\Big)^{1/2}\\&\leq C_\Delta
    \Big(\wt{\E}\E\mathbb E_{\eta} \Big\|\sum_{i=1}^M\sum_{k=0}^M\varepsilon_i\eta_k y_{ik}\Big\|_X^2\Big)^{1/2}.
\end{align*}
Since $(\eta_{k_j}\wt{\varepsilon}_j)_{j=1}^N$ again has the same distribution as
$(\wt{\varepsilon}_j)_{j=1}^N$, this is exactly the right-hand side of \eqref{eq:firstclaimtriangular}.

Next suppose that $\alpha_{1j}\geq \alpha_{2j}\geq\cdots\geq \alpha_{Mj}\geq 0$ for 
$1\leq j\leq N$, and that $\sup_{1\leq j\leq N} \alpha_{1j}\leq B$. Then
\[
    \alpha_{ij}=\int_0^B \mathbf \1_{\{\tau<\alpha_{ij}\}}\dd \tau.
\]
For each fixed $\tau \in (0,B)$, the monotonicity in $i$ implies that $\mathbf \1_{\{\tau<\alpha_{ij}\}}=\mathbf \1_{\{i\leq k_j(\tau)\}}$ for suitable integers $k_j(\tau)\in\{0,\ldots,M\}$. By \eqref{eq:firstclaimtriangular} and Minkowski's integral inequality,
\begin{align*}
    \Big(\E\wt{\E}\Big\|\sum_{i,j}\varepsilon_i\wt{\varepsilon}_j \alpha_{ij}x_{ij}\Big\|_X^2\Big)^{1/2}
    &\leq\int_0^B\Big(\E\wt{\E}\Big\|\sum_{i,j}\varepsilon_i\wt{\varepsilon}_j\mathbf \1_{\{\tau<\alpha_{ij}\}}x_{ij}\Big\|_X^2\Big)^{1/2}\dd \tau\\
    &\leq C_\Delta B\Big(\E\wt{\E}\Big\|\sum_{i,j}\varepsilon_i\wt{\varepsilon}_jx_{ij}\Big\|_X^2\Big)^{1/2}.
\end{align*}
The same estimate holds for non-negative increasing columns, by reversing the
order of the $i$-index.

Now consider real-valued coefficients. For each $j$ write    $\alpha_{ij}=\alpha_{Mj}+p_{ij}-q_{ij}$, where
\[
    p_{ij}:=\sum_{k=i}^{M-1}(\alpha_{kj}-\alpha_{k+1,j})_+,
    \qquad
    q_{ij}:=\sum_{k=i}^{M-1}(\alpha_{k+1,j}-\alpha_{kj})_+.
\]
Then $p_{\cdot j}$ and $q_{\cdot j}$ are non-negative decreasing sequences and
\[
    \sup_{1\leq j\leq N} p_{1j}+\sup_{1\leq j\leq N} q_{1j} \leq 2\sup_{1\leq j\leq N}\sum_{i=1}^{M-1}|\alpha_{ij}-\alpha_{i+1,j}|\leq 2B.
\]
The constant term $(\alpha_{Mj})_{j=1}^N$ is handled by Kahane's contraction principle in
the $\wt{\varepsilon}_j$-variable:
\[
    \Big(\E\wt{\E}\Big\|\sum_{i,j}\varepsilon_i\wt{\varepsilon}_j \alpha_{Mj}x_{ij}\Big\|_X^2\Big)^{1/2}\leq2B\Big(\E\wt{\E}\Big\|\sum_{i,j}\varepsilon_i\wt{\varepsilon}_jx_{ij}\Big\|_X^2\Big)^{1/2}.
\]
Applying the monotone case to $p$ and $q$ gives the desired estimate for real
coefficients. Finally, complex coefficients are treated by applying the real
case to their real and imaginary parts.
\end{proof}

\begin{corollary}\label{cor:BV-functions-triangular}
Let $X$ have the triangular contraction property. Let
$m_1,\ldots,m_N\colon\mathbb R_+\to\mathbb C$ be functions of bounded variation
such that
\[
    \sup_{1\leq j\leq N}
    \bigl(
        \|m_j\|_\infty+\operatorname{Var}(m_j)
    \bigr)
    \leq B .
\]
Then for all $u_1,\ldots,u_M\in\mathbb R_+$ and all
$(x_{ij})_{i,j=1}^{M,N}\subseteq X$,
\[
    \Big(
        \E\wt{\E}
        \Big\|
            \sum_{i=1}^M\sum_{j=1}^N
            \varepsilon_i\wt{\varepsilon}_j m_j(u_i)x_{ij}
        \Big\|_X^2
    \Big)^{1/2}
    \lesssim
    C_\Delta B
    \Big(
        \E\wt{\E}
        \Big\|
            \sum_{i=1}^M\sum_{j=1}^N
            \varepsilon_i\wt{\varepsilon}_jx_{ij}
        \Big\|_X^2
    \Big)^{1/2}.
\]
\end{corollary}

\begin{proof}
After simultaneously permuting the points $u_i$, the rows
$(x_{ij})_{j=1}^N$, and the corresponding Rademacher variables
$\varepsilon_i$, we may assume that
$u_1\leq\cdots\leq u_M$. Now the result follows from 
Lemma~\ref{lem:BV-triangular-contraction} with $\alpha_{ij}:=m_j(u_i)$, since
\[
    |\alpha_{Mj}|\leq\|m_j\|_\infty,    \qquad    \sum_{i=1}^{M-1}|\alpha_{ij}-\alpha_{i+1,j}|\leq \operatorname{Var}(m_j).
\]
This finishes the proof.
\end{proof}

We can now compare the two stochastic convolution conditions  $(S_p)$ and $(S_p^{\exp})$. As noted before, $(S_p)$ implies $(S_p^{\exp})$ by averaging or convexity. The converse is more delicate. We compare a normalized interval kernel with an exponential kernel of comparable scale, and show that the remaining factor is a bounded variation multiplier. The preceding corollary allows this multiplier to be handled using the triangular contraction property.

\begin{proposition}\label{prop:HpEp}
Let $X$ be a UMD Banach space with type $2$, and let $p\in  [2, \infty)$. Then the conditions
$(S_p)$ and $(S_p^{\exp})$ are equivalent.
\end{proposition}

\begin{proof}
For $\lambda>0$, one has
\[
        \int_0^\infty \sqrt{t}\,|k_\lambda'(t)|\dd t
        =\lambda^{3/2}\int_0^\infty \sqrt{t}\,e^{-\lambda t}\dd t
        =\Gamma\big(\tfrac32\big).
\]
Thus the convexity argument of \cite[Proposition 3.2]{NVW-SMR} shows that $(S_p)$ implies the $R$-boundedness of $\{S_{k_\lambda}:\lambda>0\}$. The modulation argument preceding Lemma \ref{lem:gammaRbdd} then gives $(S_p^{\exp})$.

We prove the converse in the following slightly stronger form. Let $C_0> 1$ and
$(\lambda_n)_{n\geq1}$ be an increasing sequence in $(0,\infty)$ such that
\[
     \lambda_n \to \infty,\qquad    \lambda_{n+1}\leq C_0\lambda_n,
        \qquad n\geq1,
\]
and suppose that the family $\{S_{k_{\lambda_n}}:n\geq1\}$ is $R$-bounded on
$L^p_{\mathscr F}(\Omega\times\R_+;X)$ with bound $C_{\exp}$. We prove that $X$ satisfies $(S_p)$ in five steps.

\smallskip
\emph{Step 1: Common dilation and selection of comparable scales.}
By Lemma \ref{lem:gammaRbdd}, the family $\{N_{k_{\lambda_n}}:n\geq1\}$
is $R$-bounded from $L^p(\R_+;X)$ into
$L^p(\R_+;\gamma(\R_+;X))$. For $a>0$, we have
\[
        k_{a\lambda_n}(u)=a^{1/2}k_{\lambda_n}(au).
\]
The dilation argument in the proof of Lemma \ref{lemma:timeintervalSp}, now applied on $\R_+$, therefore shows that
$\{N_{k_{a\lambda_n}}:n\geq1\}$ is $R$-bounded with a bound independent of $a>0$.

Let $N\geq1$, let $\delta_1,\ldots,\delta_N>0$, and let
$G_1,\ldots,G_N\in L^p(\R_+;X)$. Set
\[
        a:=\frac{1}{\lambda_1\max_{1\leq j\leq N}\delta_j}.
\]
For each $1\leq j\leq N$, by the definition of $a$ we can choose $n_j\geq1$ such that
$a\lambda_{n_j}\leq\delta_j^{-1}       <a\lambda_{n_j+1},$
and put
\(\theta_j:=a\lambda_{n_j}\delta_j.\)
It follows that $C_0^{-1}<\theta_j\leq1.$

\smallskip
\emph{Step 2: Factorisation by bounded-variation multipliers.}
For $u>0$, we have
\[
        h_{\delta_j}(u)=\widetilde m_j(u)k_{a\lambda_{n_j}}(u),
\]
where
\[
        \widetilde m_j(u)
        :=\theta_j^{-1/2}e^{a\lambda_{n_j} u}\1_{(0,\delta_j)}(u).
\]
Since $C_0^{-1}<\theta_j\leq1$ and
$a\lambda_{n_j}\delta_j=\theta_j$, it follows that
$\|\widetilde m_j\|_\infty\leq eC_0^{1/2}$. Moreover,
$\widetilde m_j$ is increasing on $(0,\delta_j)$ and vanishes on
$(\delta_j,\infty)$, and therefore
\[
        \operatorname{Var}(\widetilde m_j)
        \leq2eC_0^{1/2}.
\]
Thus, 
\begin{equation}\label{eq:varestequiv}
        \sup_{1\leq j\leq N}
        \bigl(
        \|\widetilde m_j\|_\infty
        +\operatorname{Var}(\widetilde m_j)
        \bigr)
        \leq3eC_0^{1/2}.
\end{equation}

\smallskip
\emph{Step 3: $R$-boundedness of the diagonal multipliers.}
Let $(\varepsilon_j)_{j=1}^N$ be a Rademacher sequence and put
\[
        Y_N:=\Rad_N^p(X)=\Big\{\sum_{j=1}^N\varepsilon_jx_j:\ x_j\in X \Big\}\subseteq L^p(\Omega_\varepsilon;X).
\]
For $u>0$, define $D(u)\in\mathscr L(Y_N)$ by
\[
        D(u)\Big(\sum_{j=1}^N\varepsilon_jx_j\Big):=\sum_{j=1}^N\varepsilon_j\widetilde m_j(u)x_j.
\]
Since $X$ is UMD, it has the triangular contraction property. By
Corollary \ref{cor:BV-functions-triangular} and \eqref{eq:varestequiv}, the family
\[
        \mathscr D:=\{D(u):u>0\}
\]
is $R$-bounded on $Y_N$, with an $R$-bound depending only on
$X$, $p$, and $C_0$. In particular, this bound is independent of
$N$ and  the numbers $\delta_j$.
Indeed, let $u_1,\ldots,u_M>0$ and write
\[
        y_i=\sum_{j=1}^N\varepsilon_jx_{ij}\in Y_N.
\]
By Kahane's inequalities and Fubini's theorem,
\[
\begin{aligned}
\Big(
\E_{\wt\varepsilon}
\Big\|
\sum_{i=1}^M\wt\varepsilon_iD(u_i)y_i
\Big\|_{Y_N}^2
\Big)^{1/2}
&\eqsim_p
\Big(
\E_{\wt\varepsilon}\E_\varepsilon
\Big\|
\sum_{i=1}^M\sum_{j=1}^N
\wt\varepsilon_i\varepsilon_j
\widetilde m_j(u_i)x_{ij}
\Big\|_X^2
\Big)^{1/2}.
\end{aligned}
\]
Corollary \ref{cor:BV-functions-triangular} and another application of
Kahane's inequalities now give the required $R$-boundedness of
$\{D(u):u>0\}$ on $Y_N$, with constants independent of $N$.

\smallskip
\emph{Step 4: Comparison of the corresponding $\gamma$-radonifying norms.}
For $t>0$, define
\begin{align*}
        \Phi_h^t(s)
        &:=
        \sum_{j=1}^N
        \varepsilon_jh_{\delta_j}(t-s)G_j(s),
        &&s\in(0,t),\\
        \Phi_k^t(s)
        &:=
        \sum_{j=1}^N
        \varepsilon_jk_{a\lambda_{n_j}}(t-s)G_j(s),
        &&s\in(0,t).
\end{align*}
The factorisation obtained in Step 2 gives
\[
        \Phi_h^t(s)=D(t-s)\Phi_k^t(s),
        \qquad 0<s<t.
\]
The $\gamma$-multiplier theorem \cite[Theorem 9.5.1]{HNVW2} therefore implies that
\begin{equation}\label{eq:step4est}
        \|\Phi_h^t\|_{\gamma(0,t;Y_N)}
        \lesssim_{p,X,C_0}
        \|\Phi_k^t\|_{\gamma(0,t;Y_N)},
        \qquad t>0.
\end{equation}
Here and below, the implicit constants are independent of
$N$, the numbers $\delta_j$, and the functions $G_j$.

\smallskip
\emph{Step 5: Conclusion.}
By the $\gamma$-Fubini isomorphism, \eqref{eq:step4est}, and the $\gamma$-Fubini isomorphism once more, we obtain
\begin{align*}
&\Big\|
\sum_{j=1}^N\varepsilon_jN_{h_{\delta_j}}G_j
\Big\|_{L^p(\Omega_\varepsilon;
L^p(\R_+;\gamma(\R_+;X)))}\lesssim_{p,X,C_0}
\Big\|
\sum_{j=1}^N\varepsilon_jN_{k_{a\lambda_{n_j}}}G_j
\Big\|_{L^p(\Omega_\varepsilon;
L^p(\R_+;\gamma(\R_+;X)))}.
\end{align*}
The $R$-boundedness obtained in Step 1 and Kahane's inequalities give
\[
\Big\|
\sum_{j=1}^N\varepsilon_jN_{h_{\delta_j}}G_j
\Big\|_{L^p(\Omega_\varepsilon;
L^p(\R_+;\gamma(\R_+;X)))}
\lesssim_{p,X,C_0} C_{\exp}\,
\Big\|
\sum_{j=1}^N\varepsilon_jG_j
\Big\|_{L^p(\Omega_\varepsilon;L^p(\R_+;X))}.
\]
Another application of Kahane's inequalities shows that
$\{N_{h_\delta}:\delta>0\}$ is $R$-bounded. By Lemma
\ref{lem:gammaRbdd}, $X$ satisfies $(S_p)$.
\end{proof}

\begin{corollary}[Extrapolation of $(S_p)$]
\label{cor:Sp-independent-p}
Let $X$ be a UMD Banach space with type $2$. If $X$ satisfies $(S_p)$ for some $p\in[2,\infty)$, then it satisfies $(S_q)$ for all $q\in(2,\infty)$. In particular, the conditions $(S_q)$, $q\in(2,\infty)$, are equivalent.
\end{corollary}

\begin{proof}
By Proposition \ref{prop:HpEp}, condition $(S_p)$ implies $(S_p^{\exp})$. The extrapolation theorem \cite[Theorem 9.1]{LoVer} then gives $(S_q^{\exp})$ for every $q\in(2,\infty)$. A second application of Proposition \ref{prop:HpEp} gives $(S_q)$.
\end{proof}

\subsection{The case \texorpdfstring{$p=2$}{p=2}}\label{ss:casep2}
Stochastic maximal $L^2$-regularity is quite rare outside the Hilbert space context. However, it can occur if one works in the real interpolation scale (see \cite[Theorem 8.6]{LoVer} and references therein). Moreover, there is a general positive SMR result if one works with a shifted scale compared to the definition in \eqref{eq:SMRest}. Suppose that $X$ is a Hilbert space, $X_1=\Dom(A)$, and $X_{1/2}=[X,X_1]_{1/2}$. If $-A$ generates a strongly continuous analytic semigroup, then in \cite[Theorem 3.13]{AVsurvey} it was shown that for every $p\in [2, \infty)$, 
\begin{align*}
\|A U_G\|_{L^p(\Omega\times (0,T);X)}\leq C \|G\|_{L^p(\Omega\times (0,T);X_{1/2})}.
\end{align*}
The unshifted version \eqref{eq:SMRest} does not always hold, since it implies a square function estimate that does not hold for all analytic semigroups (see \cite[Lemma 4.2]{AV19}). 

In this subsection we will show that $(S_2)$ only holds in Hilbert spaces.
\begin{theorem}\label{thm:HScharacS2}
Let $X$ be a UMD Banach space with type $2$. Then $X$ satisfies condition $(S_2)$ if and only if $X$ is isomorphic to a Hilbert space.
\end{theorem}

A celebrated theorem of Kwapie\'n \cite[Theorem 7.3.1]{HNVW2} asserts that a Banach space is isomorphic to a Hilbert space if and only if it has type $2$ and cotype $2$. In view of this result, and since we are already assuming that $X$ has type $2$, it therefore suffices to prove that $X$ has cotype $2$. This will follow from Lemma  \ref{lem:gammaRbdd} and the following stronger statement.

\begin{proposition}\label{prop:Ndelta}
Let $X$ be a Banach space. For $\delta>0$, define $N_\delta$, initially on the $X$-valued step functions on $\R_+$, by
\[
        (N_\delta f)(t)h
        =
        \delta^{-1/2}
        \int_{(t-\delta,t)\cap\R_+}h(s)f(s)\dd s,
        \qquad h\in L^2(\R_+).
\]
Assume that these operators extend to an $R$-bounded family from $L^2(\R_+;X)$ into \(L^2(\R_+;\gamma(L^2(\R_+),X))\), with $R$-bound at most $C$. Then $X$ has type $2$ and cotype $2$. Moreover, the cotype $2$ constant is at most $C$.
\end{proposition}

\begin{proof}

To prove that $X$ has cotype $2$ with cotype constant at most $C$, fix $x_1,\ldots,x_n\in X$. For each $j$, choose norm-one vectors $x_j^*\in X^*$ such that $\langle x_j,x_j^*\rangle=\|x_j\|$.
Let $\delta\in(0,1)$, and put \[
s_j=\delta^j,\qquad
I_j=(s_{j+1},s_j),\qquad
J=(0,s_{n+1}).
\]
The intervals $I_1,\ldots,I_n$ are pairwise disjoint and $|I_j| = (1-\delta)s_j$. Moreover,
\begin{align*}
t\in I_j\quad  \Longrightarrow \quad J\subseteq(t-s_j,t)
\end{align*}
Indeed, if $t\in I_j$, then $t<s_j$ and hence $t-s_j<0$, and also $s_{n+1}\le s_{j+1} < t$.

Define $f_j\in L^2(\R_+;X)$ by
\[f_j:=\frac{\1_J}{|J|^{1/2}} x_j.\]
By the $R$-boundedness assumption,
\begin{align*}
\E
\Big\|
\sum_{j=1}^n \varepsilon_jN_{s_j}f_j
\Big\|_{L^2(\R_+;\gamma(L^2(\R_+),X))}^2
&\le
C^2
\E
\Big\|
\sum_{j=1}^n \varepsilon_jf_j
\Big\|_{L^2(\R_+;X)}^2
=
C^2 \,
\E
\Big\|
\sum_{j=1}^n \varepsilon_jx_j
\Big\|^2.
\end{align*}
It remains to bound the left-hand side from below by $(1-\delta)\sum_{j=1}^n \|x_j\|^2$. Indeed, then we can let $\delta\downarrow 0$ and find that $X$ has cotype $2$ with constant $C$.

To estimate the left-hand side from below, for $\varepsilon = (\varepsilon_1,\dots \varepsilon_n)$ with each $\varepsilon_\ell$ unimodular, define $G_\varepsilon:\R_+\to \gamma(L^2(\R_+),X)$ by
\[
G_{\varepsilon}(t):=\sum_{\ell=1}^n \varepsilon_\ell(N_{s_\ell}f_\ell)(t).
\]
For every $R\in\gamma(L^2(\R_+),X)$ and $x^*\in X^*$,
\[
\|R^*x^*\|_{L^2(\R_+)}
\leq
\|R\|_{\gamma(L^2(\R_+),X)}\|x^*\|.
\]
Now let $(\varepsilon_\ell)_{\ell=1}^n$ be a Rademacher sequence. Applying the definition of $G_\varepsilon$ pointwise and taking square expectations, for $t\in I_j$ we obtain
\[
\E\|G_{\varepsilon}(t)\|_{\gamma(L^2(\R_+),X)}^2
\geq
\E\|G_{\varepsilon}(t)^*x_j^*\|_{L^2(\R_+)}^2.
\]
An explicit calculation gives
\[
G_{\varepsilon}(t)^*x_j^*
=
\sum_{\ell=1}^n
\varepsilon_\ell s_\ell^{-1/2} |J|^{-1/2}
\1_{J\cap(t-s_\ell,t)}
\langle x_\ell,x_j^*\rangle.
\]
By orthogonality of the Rademacher variables,
\begin{align*}
\E\|G_{\varepsilon}(t)^*x_j^*\|_{L^2(\R_+)}^2
&=
\sum_{\ell=1}^n
\frac{|\langle x_\ell,x_j^*\rangle|^2}{s_\ell |J|}
\bigl|J\cap(t-s_\ell,t)\bigr|\geq
\frac{|\langle x_j,x_j^*\rangle|^2}{s_j |J|}
\bigl|J\cap(t-s_j,t)\bigr|.
\end{align*}
Recalling that $\langle x_j,x_j^*\rangle=\|x_j\|$ and that $t\in I_j$ implies $J\subseteq(t-s_j,t)$, we find that
\[
\E\|G_{\varepsilon}(t)\|_{\gamma(L^2(\R_+),X)}^2
\geq
\frac{1}{s_j}\|x_j\|^2,
\qquad t\in I_j.
\]
Integrating over the pairwise disjoint intervals $I_j$, we obtain
\begin{align*}
\E
\Bigl\|
\sum_{j=1}^n \varepsilon_jN_{s_j}f_j
\Bigr\|_{L^2(\R_+;\gamma(L^2(\R_+),X))}^2
 &=
\int_{\R_+}\E\|G_{\varepsilon}(t)\|_{\gamma(L^2(\R_+),X)}^2\dd t
\\ & \geq
\sum_{j=1}^n
\int_{I_j}
\frac{1}{s_j}\|x_j\|^2\dd t
=
(1-\delta)
\sum_{j=1}^n\|x_j\|^2.
\end{align*}

Finally, combining the upper and lower estimates, we obtain
\[
        (1-\delta)\sum_{j=1}^n\|x_j\|^2
        \leq
        C^2\E\Big\|\sum_{j=1}^n\varepsilon_jx_j\Big\|^2.
\]
The proof is completed by letting $\delta\downarrow0$.
\end{proof}

\begin{proof}[Proof of Theorem \ref{thm:HScharacS2}]
Suppose first that $X$ is isomorphic to a Hilbert space. The domain and codomain of the stochastic convolution operators in Definition \ref{def:Rbddconv}, with $p=2$, are then isomorphic to Hilbert spaces. In Hilbert spaces, uniform boundedness is equivalent to $R$-boundedness. Thus, the uniform boundedness estimate \(\sup_{\delta>0}\|S_{h_\delta}\|<\infty\) implies that $X$ satisfies $(S_2)$.

Conversely, suppose that $X$ is a UMD space with type $2$ satisfying $(S_2)$. By Lemma \ref{lem:gammaRbdd}, the family \(\{N_\delta:\delta>0\}\) is $R$-bounded, hence Proposition \ref{prop:Ndelta} shows that $X$ has cotype $2$. Since $X$ has type $2$ by assumption, Kwapie\'n's theorem implies that $X$ is isomorphic to a Hilbert space.
\end{proof}

\subsection{Schauder multiplier example}\label{ss:Schauder}
We next show that the stochastic convolution condition $(S_p)$ is not just  a technical assumption in the proof of positive results. We will show that it is necessary for a simple diagonal operator acting on the Rademacher space
\[
\Rad^p(X):=
\overline{\Bigl\{
\sum_{n=1}^N\varepsilon_nx_n:
N\geq1,\ x_1,\ldots,x_N\in X
\Bigr\}}^{L^p(\Omega_\varepsilon;X)}.
\] 
Indeed, we will discuss the operator $A$ on $\Rad^p(X)$  defined by
\begin{equation}\label{eq:Arad}
A (x_n)_{n\geq 1} = (2^n x_n)_{n\geq 1}
\end{equation}
with its natural domain
\[\Dom(A) = \bigl\{(x_n)_{n\geq 1}\in \Rad^p(X): (2^n x_n)_{n\geq 1}\in \Rad^p(X)\bigr\}.\]

Let us first record the following: 

\begin{proposition}\label{prop:HinftyRad}
Let $X$ be a Banach space and let $p\in[1,\infty)$. Then the
operator $A$ defined by
\[
        A\Big(\sum_{n\geq1}\varepsilon_nx_n\Big)
        :=
        \sum_{n\geq1}2^n\varepsilon_nx_n
\]
on $\Rad^p(X)$, with its natural domain, is sectorial and has a
bounded $H^\infty$-calculus of angle $0$.
\end{proposition}

\begin{proof} It is routine to check that finite Rademacher sums form a core for $A$. For
$\lambda\notin[0,\infty)$, $\lambda - A$ is invertible and its resolvent is given on such sums by
\[
        R(\lambda,A)
        \Big(\sum_{n=1}^N\varepsilon_nx_n\Big)
        =
        \sum_{n=1}^N
        \frac{1}{\lambda-2^n}\varepsilon_nx_n.
\]
For every $\theta>0$,
\[
        \sup_{\lambda\notin\overline{\Sigma_\theta}}
        \sup_{n\geq1}
        \Big|\frac{\lambda}{\lambda-2^n}\Big|
        <\infty.
\]
The Kahane contraction principle therefore shows that $A$ is sectorial of angle $0$.

Similarly, for $\varphi\in H^\infty(\Sigma_\theta)$, we may define the bounded operator $\varphi(A)$ by
\[
        \varphi(A)
        \Big(\sum_{n=1}^N\varepsilon_nx_n\Big)
        =
        \sum_{n=1}^N\varepsilon_n\varphi(2^n)x_n.
\]
By Cauchy's theorem, for functions $\varphi \in H^1(\Sigma_\theta)\cap H^\infty(\Sigma_\theta)$, $\varphi(A)$ is given by the Dunford calculus, i.e.,
$$ \varphi(A) = \frac1{2\pi i} \int_{\partial \Sigma_\nu} \varphi(z) R(z,A)\dd z,$$
for any $\nu\in (0,\theta)$. The norm of $\varphi(A)$ can be directly estimated as
\[
        \|\varphi(A)\|_{\mathscr L(\Rad^p(X))}
        \le
        \sup_{n\geq1}|\varphi(2^n)|
        \leq
        \|\varphi\|_{H^\infty(\Sigma_\theta)}.
\]
By \cite[Definition 10.2.20]{HNVW2}, this shows that $A$ has a bounded $H^\infty$-calculus of angle $0$.
\end{proof}

The preceding observations lead to the following characterisation.
\begin{theorem}\label{thm:SMRRad}
Let $X$ be a UMD Banach space with type $2$ and let $p\in [2, \infty)$. Let $A$ be as in \eqref{eq:Arad}. Let $T\in (0,\infty)$. Then the following are equivalent:
\begin{enumerate}[label={\rm(\arabic*)}, leftmargin=*]
\item\label{it1:SMRRad} The operator $A$ has stochastic maximal $L^p$-regularity on $(0,T)$;
\item\label{it2:SMRRad} The operator $A$ has stochastic maximal $L^p$-regularity on $\R_+$;
\item\label{it3:SMRRad} The $R$-boundedness condition $(S_p)$ holds.
\end{enumerate}
\end{theorem}

\begin{proof}
Write $Y:=\Rad^p(X)$. The equivalence of \ref{it1:SMRRad} and \ref{it2:SMRRad} follows from \cite{AV19} and the fact that the semigroup $e^{-tA}$ is uniformly exponentially stable.

\ref{it3:SMRRad} $\Rightarrow$ \ref{it2:SMRRad}: By Proposition \ref{prop:HpEp}, $(S_p^{\exp})$ holds. By Fubini's theorem we can reformulate this on $Y$, and this immediately implies that $A$ has stochastic maximal $L^p$-regularity on $\R_+$. Indeed,
\[
\begin{aligned}
    \int_0^t A^{1/2}e^{-(t-s)A}\sum_{n\geq 1}\varepsilon_n G_n(s)\dd W(s)
    &=
    \sum_{n\geq 1}\varepsilon_n
    \int_0^t 2^{n/2}e^{-2^n(t-s)}G_n(s)\dd W(s)       \\
    &=
    \sum_{n\geq 1}\varepsilon_n S_{k_{2^n}}G_n(t).
\end{aligned}
\]
Thus, the stochastic maximal $L^p$-regularity estimate for $A$ on $\R_+$ takes the form
\[
    \Big\|
        \sum_{n\geq 1}\varepsilon_n S_{k_{2^n}} G_n
    \Big\|_{L^p(\Omega_\varepsilon;L^p(\Omega\times\R_+;X))}
    \lesssim
    \Big\|
        \sum_{n\geq 1}\varepsilon_n G_n
    \Big\|_{L^p(\Omega_\varepsilon;
        L^p(\Omega\times\R_+;X))},
\]
and this bound holds by $(S_p^{\exp})$.

\smallskip

\ref{it2:SMRRad} $\Rightarrow$ \ref{it3:SMRRad}:
Let $n_1,\ldots,n_N$ be pairwise distinct and let $G_1,\ldots,G_N$ be adapted simple $X$-valued processes. Applying stochastic maximal $L^p$-regularity to
\(G:=\sum_{j=1}^N\varepsilon_{n_j}G_j\)
gives
\[
\Big\|
\sum_{j=1}^N
\varepsilon_{n_j}S_{k_{2^{n_j}}}G_j
\Big\|_{L^p(\Omega_\varepsilon;
L^p(\Omega\times\R_+;X))}
\lesssim
\Big\|
\sum_{j=1}^N
\varepsilon_{n_j}G_j
\Big\|_{L^p(\Omega_\varepsilon;
L^p(\Omega\times\R_+;X))}.
\]
Thus the estimate required for $R$-boundedness holds for distinct members of the family $  \{S_{k_{2^n}}:n\geq1\}.$
Allowing repetitions does not change the conclusion: one groups equal indices and introduces an additional independent Rademacher sequence. Hence the full dyadic family is $R$-bounded.
Applying the stronger form of Proposition \ref{prop:HpEp} mentioned in its proof, with
$\lambda_n=2^n$, now gives condition $(S_p)$.
\end{proof}

As a consequence we obtain the following, which  gives an alternative route to Corollary \ref{cor:main}.

\begin{corollary}\label{cor:noSMR}
Let $p\in [2, \infty)$ and $r\in (2, \infty)$, and let $A$ be the operator on $\Rad^p(\ell^{2}(\ell^r))$ defined by \eqref{eq:Arad}, i.e.,
\begin{equation*}
A (x_n)_{n\geq 1} = (2^n x_n)_{n\geq 1}
\end{equation*}
with its natural domain. Let either $I= (0,T)$ with $T\in (0,\infty)$ or $I = \R_+$. Then $A$ has a bounded $H^\infty$-calculus of angle $0$, but $A$ does not have stochastic maximal $L^p$-regularity on $I$.
\end{corollary}

 \begin{proof}
 The assertion on the $H^\infty$-calculus follows from Proposition \ref{prop:HinftyRad}. However, as in \cite[Theorem 8.2]{NVW-Rbounded} (where only $r=4$ was considered)
 one sees that $\ell^{2}(\ell^r)$ does not satisfy $(S_p)$ for any $p\in (2,\infty)$.
 If we had $(S_2)$, Theorem \ref{thm:HScharacS2} would imply that $\ell^2(\ell^r)$ is isomorphic to a Hilbert space, which is not the case. Therefore, for all $p\in [2,\infty)$, by Theorem \ref{thm:SMRRad}, the operator $A$ cannot have stochastic maximal $L^p$-regularity.
 \end{proof}

\subsection{The Laplacian}\label{ss:Laplace}
In this final subsection we relate the abstract stochastic convolution condition to the Laplacian $A = -\Delta$ on $L^q(\mathbb R^d;X)$, with $X$ a UMD space and $1<q<\infty$. Under these assumptions, $A$ has a bounded $H^\infty$-calculus of  angle $0$ by \cite[Theorem 10.2.25]{HNVW2},
and stochastic maximal $L^p$-regularity follows if $L^q(\mathbb R^d;X)$ satisfies $(S_p)$ for $p \geq 2 $. The theorem below shows that, in the UMD type $2$ setting with $q\in (2,\infty)$, stochastic maximal regularity of $A$ also forces the same $R$-boundedness condition $(S_p)$. The proof uses a large scale plane-wave argument to extract one-dimensional exponential stochastic convolutions from the heat semigroup.

We start with a lemma that allows us to localise the Laplacian to a single frequency.
\begin{lemma}
\label{lem:large-scale-plane-waves}
Let $X$ be a Banach space, let $p\in[2,\infty)$,
$q\in[1,\infty)$, $\xi \in \R^d\setminus\cbrace{0}$ and let
$\phi\in\mc S(\R^d)$ be non-zero with compact Fourier support.  For
$R>0$ set
$
  \phi_R(y)=R^{-d/q}\phi(y/R)
$
and
define
\[
  T_{R,\xi}(t,y)
  :=
  \brb{(-\Delta)^{1/2}e^{t\Delta}-|\xi|e^{-t|\xi|^2}}(e^{i \ip{\xi,\cdot} }\phi_R)(y), \qquad (t,y) \in \R_+\times \R^d.
\]
For a bounded interval $I \subseteq \R_+$ and $x \in X$ we have
\[
  \lim_{R\to\infty}
  \nrm{
    t\mapsto
    \bracb{
      s\mapsto
      \1_I(t-s) T_{R,\xi}(s,\cdot)x
    }
  }_{L^p(\R_+;\gamma(\R_+;L^q(\R^d;X)))}
  =
  0.
\]
\end{lemma}

\begin{proof}
For $s \in \R_+$ write
\[
  m_s(\zeta):=2\pi|\zeta|e^{-4\pi^2s|\zeta|^2}, \qquad \zeta \in \R^d,
\]
and note that
\[
\begin{aligned}
  T_{R,\xi}(s,y)
  &=
  \mc F^{-1}_\zeta
  \Big(
    \bracb{m_s(\zeta)-m_s(\tfrac{\xi}{2\pi})}
    R^{d(1-1/q)}\widehat\phi(R(\zeta-\tfrac{\xi}{2\pi}))
  \Big)(y)
  \\
  &=
  R^{-d/q}e^{i\xi\cdot y}
  \mc F^{-1}_\eta \brB{\bracb{m_s(\tfrac{\xi}{2\pi}+\tfrac\eta{R})-m_s(\tfrac{\xi}{2\pi})}\widehat\phi(\eta)}(y/R).
\end{aligned}
\]
Thus, defining
\[
  e_{R,s,\xi}(\eta)
  :=
  \bracb{m_s(\tfrac{\xi}{2\pi}+\tfrac\eta{R})-m_s(\tfrac{\xi}{2\pi})}\widehat\phi(\eta), \qquad \eta \in \R^d,
\]
by a change of variables we have, for any bounded interval $I\subseteq \R_+$
\begin{align}\label{eq:Dtoe}
  \nrmB{
    \brB{\int_I |T_{R,\xi}(s,\cdot)|^2\dd s}^{1/2}
  }_{L^q(\R^d)}
  =
  \nrmB{
    \brB{\int_I |\widecheck{e_{R,s,\xi}}|^2\dd s}^{1/2}
  }_{L^q(\R^d)} .
\end{align}
Let $r>0$ be such that $\supp\widehat\phi\subseteq B_r(0)$.  Then, if
$R\geq \tfrac{4\pi r}{|\xi|}$ and $\eta\in B_r(0)$,
\[
\cbraceb{\tfrac{\xi}{2\pi}+\tfrac{\theta\eta}{R}: 0\leq \theta\leq1}
\subseteq
\cbraceb{\zeta \in \R^d:
  \tfrac{|\xi|}{4\pi}\leq |\zeta|\leq \tfrac{3|\xi|}{4\pi}},
\]
and for every
$\beta \in \N^d$,
\begin{equation*}
  |\partial_\zeta^\beta m_s(\zeta)|
  \leq
  C_{\beta,\xi}(1+s)^{|\beta|+1}e^{-c_\xi s},
  \qquad s>0,\,  \tfrac{|\xi|}{4\pi}\leq |\zeta|\leq \tfrac{3|\xi|}{4\pi}.
\end{equation*}
Therefore, by the mean value theorem,
\[
  \sum_{|\beta|\leq d+1}
  \nrm{\partial_\eta^\beta e_{R,s,\xi}(\cdot)}_{L^1(\R^d)}
  \leq
  \frac{C_{d,\xi,\phi}}{R}
  (1+s)^{d+2}e^{-c_\xi s}.
\]
Taking the Fourier inverse, this yields
\[
  |\widecheck{e_{R,s,\xi}}(x)|
  \leq
  \frac{C_{d,\xi,\phi}}{R}
  (1+s)^{d+2}e^{-c_\xi s}(1+|x|)^{-d-1}.
\]
In combination with \eqref{eq:Dtoe}, for any bounded interval $I\subseteq \R_+$ this gives
\begin{equation}\label{eq:step1}
  \nrmB{
    \brB{\int_I |T_{R,\xi}(s,\cdot)|^2\dd s}^{1/2}
  }_{L^q(\R^d)}
  \lesssim_{d,\xi,\phi}
  \frac{1}{R}
  \brB{\int_I k_\xi(s)^2\dd s}^{1/2},
 \end{equation}
where $k_\xi(s)=(1+s)^{d+2}e^{-c_\xi s}$.

Now fix $t\in\R_+$ and set
$
  I_t:
   =
  \cbrace{s\in\R_+: t-s\in I}.
$
Noting that $$(s,y)\mapsto \1_I(t-s)T_{R,\xi}(s,y)x$$ takes its values in the one-dimensional subspace $\spn{x} \subseteq X$ and using \cite[Proposition 9.3.2]{HNVW2} and \eqref{eq:step1} with $I=I_t$, we obtain
\[
\begin{aligned}
  \nrm{
    s\mapsto \1_I(t-s)T_{R,\xi}(s,\cdot)x
  }_{\gamma(\R_+;L^q(\R^d;X))}
  &
  \lesssim_q
  \nrm{x}_X
  \nrmB{
    \brB{\int_{I_t}|T_{R,\xi}(s,\cdot)|^2\dd s}^{1/2}
  }_{L^q(\R^d)}
  \\
  &
  \lesssim_{d,q,\xi,\phi}
  \frac{\nrm{x}_X}{R}
  \brB{\int_0^\infty \1_I(t-s)k_\xi(s)^2\dd s}^{1/2}.
\end{aligned}
\]
The result now follows by letting $R\to\infty$.
\end{proof}

For $q\in (1, \infty)$, the operator $A$ has a bounded $H^\infty$-calculus of angle $0$ if and only if $X$ is a UMD space (see \cite[Theorem 10.2.25, Remark 10.2.26]{HNVW2}). For maximal $L^p$-regularity a similar equivalence was established in \cite[Theorem 17.4.1]{HNVW3}. For $R$-sectoriality, a characterisation in terms of the UMD property is only known for Banach function spaces (see \cite[Theorem 2.4.9]{KLW19}).

The next result gives a complete characterisation in the range $p,q\in(2,\infty)$.

\begin{theorem}[Stochastic maximal regularity of the Laplacian]
\label{thm:large-scale-plane-wave-extraction}
Let $X$ be a UMD Banach space with type $2$, let $p,q\in(2,\infty)$, and let $T\in(0,\infty)$. For the Laplacian $A=-\Delta$ on $L^q(\R^d;X)$, the following assertions are
equivalent:
\begin{enumerate}[label={\rm(\arabic*)}, leftmargin=*]
\item\label{it1:large-scale-plane-wave-extraction}
$A$ has stochastic maximal $L^p$-regularity on $(0,T)$;
\item\label{it2:large-scale-plane-wave-extraction}
$A$ has stochastic maximal $L^p$-regularity on $\R_+$;
\item\label{it3:large-scale-plane-wave-extraction}
$X$ satisfies condition $(S_p)$.
\end{enumerate}
\end{theorem}

\begin{proof}
Let $C_I(A)$ denote the constant in the equivalent formulation of stochastic maximal $L^p$-regularity in \eqref{eq:SMRgamma}.

\smallskip
\ref{it1:large-scale-plane-wave-extraction} $\Leftrightarrow$ \ref{it2:large-scale-plane-wave-extraction}: \
For $\lambda>0$, define
\[
        (D_\lambda f)(x):=\lambda^{d/q}f(\lambda x),
        \qquad f\in L^q(\R^d;X).
\]
Then $D_\lambda$ is an invertible isometry, with inverse $D_\lambda^{-1} = D_{\lambda^{-1}}$, and we have \(D_{\lambda^{-1}}AD_\lambda=\lambda^2A\). It follows by similarity invariance that
$ C_I(A)=C_I(\lambda^2A).$
On the other hand, from the formulation \eqref{eq:SMRgamma} one can deduce that  $C_I(\lambda^2A)=C_{\lambda^2I}(A)$. Thus,
\begin{align}\label{eq:scalingSMR}
        C_{\lambda^2I}(A)=C_I(A),
        \qquad \lambda>0.
\end{align}

Now assume that \ref{it1:large-scale-plane-wave-extraction} holds. Taking $I=(0,T)$ in \eqref{eq:scalingSMR}, we see that $A$ has stochastic maximal $L^p$-regularity on every interval $(0,T')$, with a constant independent of $T'>0$. Applying these estimates to the restrictions of $G\in L^p(\R_+;L^q(\R^d;X))$ and letting $T'\to\infty$, monotone convergence gives stochastic maximal $L^p$-regularity on $\R_+$. This proves the implication \ref{it1:large-scale-plane-wave-extraction} $\Rightarrow$ \ref{it2:large-scale-plane-wave-extraction}. The converse implication follows immediately by restriction.

\smallskip
\ref{it3:large-scale-plane-wave-extraction} $\Rightarrow$ \ref{it2:large-scale-plane-wave-extraction}: \
Assume that \ref{it3:large-scale-plane-wave-extraction} holds. By Proposition \ref{prop:HpEp}, the space $X$ satisfies $(S_p^{\exp})$. Since $p,q\in(2,\infty)$, the extrapolation result \cite[Theorem 9.1]{LoVer} implies that $X$ also satisfies $(S_q^{\exp})$. By Fubini's theorem, the space $L^q(\R^d;X)$ satisfies $(S_q^{\exp})$, and a second application of the same extrapolation result shows that it satisfies $(S_p^{\exp})$.

The space $L^q(\R^d;X)$ is a UMD space with type $2$, and the Laplacian on this space has a bounded $H^\infty$-calculus of angle $0$. Theorem \ref{thm:survey} therefore implies that $A$ has stochastic maximal $L^p$-regularity on $\R_+$. This proves the implication \ref{it3:large-scale-plane-wave-extraction} $\Rightarrow$ \ref{it2:large-scale-plane-wave-extraction}.

\smallskip
\ref{it2:large-scale-plane-wave-extraction} $\Rightarrow$ \ref{it3:large-scale-plane-wave-extraction}: \  Assume that \ref{it2:large-scale-plane-wave-extraction} holds. We prove first that the family of exponential stochastic convolution operators is $R$-bounded. By density, it suffices to prove the corresponding estimate for functions $f_1,\ldots,f_n$ which are finite linear combinations of functions of the form $\1_Ix$, where $I\subseteq\R_+$ is a bounded interval and $x\in X$. This will be done in three steps.

Fix a non-zero function $\phi\in\mc S(\R^d)$ with compact Fourier support and, for $R>0$, set
\[
        \phi_R(y):=R^{-d/q}\phi(y/R).
\]

{\em Step 1: a modulation estimate.} For $f_1,\ldots,f_n\in L^p(\R_+;X)$ and $\xi_1,\ldots,\xi_n\in\R^d$, the Kahane--Khintchine inequalities, Fubini's theorem, and the Kahane contraction principle give
\begin{equation}\label{eq:modulation-estimate}
\begin{aligned}
&\Big\|
(t,y)\mapsto
\phi_R(y)\sum_{j=1}^n
\varepsilon_je^{i\xi_j\cdot y}f_j(t)
\Big\|_{L^p(\Omega_\varepsilon;
L^p(\R_+;L^q(\R^d;X)))}
\\
 &\qquad\eqsim_{p,q}
 \|\phi\|_{L^q(\R^d)}
 \Big\|
 \sum_{j=1}^n\varepsilon_jf_j
 \Big\|_{L^p(\Omega_\varepsilon;L^p(\R_+;X))}.
 \end{aligned}
 \end{equation}
 Indeed, for fixed $t\in\R_+$, the Kahane--Khintchine inequalities, first in the Banach space $L^q(\R^d;X)$ and then in $X$, give
 \begin{align*}
 & \Big(
\E_\varepsilon
\Big\|
\phi_R(\cdot)
\sum_{j=1}^n
\varepsilon_je^{i\xi_j\cdot(\,\cdot\,)}f_j(t)
\Big\|_{L^q(\R^d;X)}^p
\Big)^{1/p}
\\
&\qquad\eqsim_{p,q}
\Big(
\int_{\R^d}
|\phi_R(y)|^q
\E_\varepsilon
\Big\|
\sum_{j=1}^n
\varepsilon_je^{i\xi_j\cdot y}f_j(t)
\Big\|_X^q
\dd y
\Big)^{1/q}
\\
&\qquad=
\|\phi_R\|_{L^q(\R^d)}
\Big(
\E_\varepsilon
\Big\|
\sum_{j=1}^n\varepsilon_jf_j(t)
\Big\|_X^q
\Big)^{1/q}
\eqsim_{p,q}
\|\phi\|_{L^q(\R^d)}
\Big(
\E_\varepsilon
\Big\|
\sum_{j=1}^n\varepsilon_jf_j(t)
\Big\|_X^p
\Big)^{1/p}.
\end{align*}
Here the middle equality follows since the unimodular numbers do not change the distribution of the complex Rademacher random variables. Integration with respect to
$t$ gives \eqref{eq:modulation-estimate}.

\emph{Step 2: applying stochastic maximal regularity.}
Let $\lambda_1,\ldots,\lambda_n>0$, and choose
$\xi_1,\ldots,\xi_n\in\R^d$ such that
\[
        |\xi_j|^2=\lambda_j,
        \qquad 1\leq j\leq n.
\]
Define
\[
        F_R(t,y)
        :=
        \sum_{j=1}^n
        \varepsilon_je^{i\xi_j\cdot y}\phi_R(y)f_j(t),
        \qquad (t,y)\in\R_+\times\R^d.
\]
Applying \eqref{eq:SMRgamma} to $F_R$, pointwise in $\Omega_\varepsilon$, and taking the $L^p(\Omega_\varepsilon)$-norm, we obtain
\begin{equation}\label{eq:FRbound}
\begin{aligned}
&\Big\|
t\mapsto
\Big[
s\mapsto
\1_{\{0<s<t\}}
(-\Delta)^{1/2}e^{s\Delta}F_R(t-s)
\Big]
\Big\|_{L^p(\Omega_\varepsilon;
L^p(\R_+;\gamma(\R_+;L^q(\R^d;X))))}
\\
&\qquad\lesssim
\|F_R\|_{L^p(\Omega_\varepsilon;
L^p(\R_+;L^q(\R^d;X)))}
\lesssim_{p,q,\phi}
\Big\|
\sum_{j=1}^n\varepsilon_jf_j
\Big\|_{L^p(\Omega_\varepsilon;L^p(\R_+;X))},
\end{aligned}
\end{equation}
where the second estimate follows from \eqref{eq:modulation-estimate}.

\smallskip
\emph{Step 3: passing to the plane-wave limit.}
By the definition of $T_{R,\xi_j}$,
\begin{equation*}
\begin{aligned}
(-\Delta)^{1/2}e^{s\Delta}
\big(e^{i\xi_j\cdot(\,\cdot\,)}\phi_R\big)
&=
|\xi_j|e^{-s|\xi_j|^2}
e^{i\xi_j\cdot(\,\cdot\,)}\phi_R
+
T_{R,\xi_j}(s,\cdot)
\\
&=
\lambda_j^{1/2}e^{-\lambda_js}
e^{i\xi_j\cdot(\,\cdot\,)}\phi_R
+
T_{R,\xi_j}(s,\cdot).
\end{aligned}
\end{equation*}
Since each $f_j$ is a finite linear combination of functions $\1_Ix$ with $I$ bounded, Lemma \ref{lem:large-scale-plane-waves}, applied to each $\xi_j$, shows that the contribution of the terms \(T_{R,\xi_j}(s,\cdot)f_j(t-s)\) to the left-hand side of \eqref{eq:FRbound} converges to zero as $R\to\infty$.

By the $\gamma$-Fubini isomorphism \cite[Theorem 9.4.8]{HNVW2}, we have the isomorphism of Banach spaces
\(\gamma(\R_+;L^q(\R^d;X)) \eqsim L^q(\R^d;\gamma(\R_+;X))\).
 Hence we may apply \eqref{eq:modulation-estimate} with $X$ replaced by $\gamma(\R_+;X)$. Together with
\[
\|\phi_R f\|_{L^q(\R^d;\gamma(\R_+;X))}
=
\|\phi\|_{L^q(\R^d)}
\|f\|_{\gamma(\R_+;X)},
\]
this allows us to pass to the limit $R\to\infty$ in
\eqref{eq:FRbound}.
Accordingly we obtain
\begin{equation}\label{eq:lambdajfjbound}
\begin{aligned}
&\Big\|
t\mapsto
\Big[
s\mapsto
\sum_{j=1}^n
\varepsilon_j\1_{\{0<s<t\}}
\lambda_j^{1/2}e^{-\lambda_js}f_j(t-s)
\Big]
\Big\|_{L^p(\Omega_\varepsilon;
L^p(\R_+;\gamma(\R_+;X)))}
\\
&\qquad\lesssim_{p,q,\phi}
\Big\|
\sum_{j=1}^n\varepsilon_jf_j
\Big\|_{L^p(\Omega_\varepsilon;L^p(\R_+;X))}.
\end{aligned}
\end{equation}

For each $t>0$, the change of variables $s\mapsto t-s$ is an isometry on $L^2(0,t)$, and hence induces an isometry on the corresponding $\gamma$-radonifying spaces. Thus, by the Kahane–Khintchine inequality, \eqref{eq:lambdajfjbound}  is equivalent to the $R$-boundedness estimate for the family \(\{N_{k_\lambda}:\lambda>0\}\). Hence by Lemma \ref{lem:gammaRbdd}, the family \(\{S_{k_\lambda}:\lambda>0\}\) is $R$-bounded. The observation following Definition \ref{def:Rbddconv} therefore implies $(S_p^{\exp})$, and Proposition \ref{prop:HpEp} gives $(S_p)$. This completes the proof of the implication
\ref{it2:large-scale-plane-wave-extraction} $\Rightarrow$ \ref{it3:large-scale-plane-wave-extraction}.
\end{proof}

\begin{remark}[Endpoint cases]
\label{rem:large-scale-plane-wave-endpoints}
The scaling argument and the plane-wave argument remain valid for $p,q\in[2,\infty)$. Thus, if $p=2$ or $q=2$, assertions
\ref{it1:large-scale-plane-wave-extraction} and \ref{it2:large-scale-plane-wave-extraction} are equivalent and imply \ref{it3:large-scale-plane-wave-extraction}.

If $p=q=2$, the converse implication also holds. Indeed, if $X$ satisfies $(S_2)$, then it satisfies $(S_2^{\exp})$ by Proposition \ref{prop:HpEp}. Fubini's theorem shows directly that $L^2(\R^d;X)$ satisfies $(S_2^{\exp})$, and Theorem \ref{thm:survey} gives stochastic maximal $L^2$-regularity of the Laplacian. Hence all three assertions are equivalent when $p=q=2$. If exactly one of $p$ and $q$ equals $2$, the extrapolation argument does not apply at the endpoint, and the proof above does not give the converse implication.
\end{remark}

\section*{AI disclosure statement}

GPT 5.5 Pro, GPT 5.6 Pro, and Codex were used during the preparation of this paper to explore proof strategies, organise intermediate LaTeX drafts, and check elementary estimates. The authors reviewed all mathematical arguments and are responsible for the content of the paper.

\bibliographystyle{plain}
\bibliography{literature}

\def\polhk#1{\setbox0=\hbox{#1}{\ooalign{\hidewidth \lower1.5ex\hbox{`}\hidewidth\crcr\unhbox0}}} \def\cprime{$'$}
\begin{thebibliography}{10}

\bibitem{agresti2023primitive}
A.~Agresti.
\newblock The primitive equations with rough transport noise: Global well-posedness and regularity.
\newblock {\em arXiv preprint arXiv:2310.01193}, 2023.

\bibitem{AV19}
A.~Agresti and M.C. Veraar.
\newblock Stability properties of stochastic maximal {$L^p$}-regularity.
\newblock {\em J. Math. Anal. Appl.}, 482(2):123553, 35, 2020.

\bibitem{AVsurvey}
A.~Agresti and M.C. Veraar.
\newblock Nonlinear {SPDEs} and maximal regularity: an extended survey.
\newblock {\em NoDEA, Nonlinear Differ. Equ. Appl.}, 32(6):150, 2025.
\newblock Id/No 123.

\bibitem{DHMPT}
R.~Danchin, M.~Hieber, P.B. Mucha, and P.~Tolksdorf.
\newblock {\em Free boundary problems via {Da} {Prato}-{Grisvard} theory}, volume 1578 of {\em Mem. Am. Math. Soc.}
\newblock Providence, RI: American Mathematical Society (AMS), 2025.

\bibitem{Fack14}
S.~Fackler.
\newblock The {Kalton}-{Lancien} theorem revisited: maximal regularity does not extrapolate.
\newblock {\em J. Funct. Anal.}, 266(1):121--138, 2014.

\bibitem{Fackl16}
S.~Fackler.
\newblock Maximal regularity: positive counterexamples on {UMD}-{Banach} lattices and exact intervals for the negative solution of the extrapolation problem.
\newblock {\em Proc. Am. Math. Soc.}, 144(5):2015--2028, 2016.

\bibitem{HNVW1}
T.P. Hyt\"onen, J.M.A.M.~van Neerven, M.C. Veraar, and L.~Weis.
\newblock {\em Analysis in {B}anach spaces. {V}ol. {I}. {M}artingales and {L}ittlewood-{P}aley theory}, volume~63 of {\em Ergebnisse der Mathematik und ihrer Grenzgebiete. 3. Folge.}
\newblock Springer, 2016.

\bibitem{HNVW2}
T.P. Hyt\"onen, J.M.A.M.~van Neerven, M.C. Veraar, and L.~Weis.
\newblock {\em Analysis in {B}anach spaces. {V}ol. {II}. {P}robabilistic {M}ethods and {O}perator {T}heory}, volume~67 of {\em Ergebnisse der Mathematik und ihrer Grenzgebiete. 3. Folge.}
\newblock Springer, 2017.

\bibitem{HNVW3}
T.P. Hyt{\"o}nen, J.M.A.M.~van Neerven, M.C. Veraar, and L.~Weis.
\newblock {\em Analysis in Banach Spaces: Volume III: Harmonic Analysis and Spectral Theory}, volume~76.
\newblock Springer Nature, 2023.

\bibitem{KaLa02}
N.~J. Kalton and G.~Lancien.
\newblock {{\(L_p\)}}-maximal regularity on {Banach} spaces with a {Schauder} basis.
\newblock {\em Arch. Math.}, 78(5):397--408, 2002.

\bibitem{KaLa}
N.J. Kalton and G.~Lancien.
\newblock A solution to the problem of {$L\sp p$}-maximal regularity.
\newblock {\em Math. Z.}, 235(3):559--568, 2000.

\bibitem{KLW19}
N.J. Kalton, E.~Lorist, and L.~Weis.
\newblock {\em Euclidean structures and operator theory in {Banach} spaces}, volume 1433 of {\em Mem. Am. Math. Soc.}
\newblock Providence, RI: American Mathematical Society (AMS), 2023.

\bibitem{KWcalc}
N.J. Kalton and L.W. Weis.
\newblock The {$H^\infty$}-calculus and sums of closed operators.
\newblock {\em Math. Ann.}, 321(2):319--345, 2001.

\bibitem{LT79}
J.~Lindenstrauss and L.~Tzafriri.
\newblock {\em Classical {B}anach spaces. {II}}, volume~97 of {\em Ergebnisse der Mathematik und ihrer Grenzgebiete}.
\newblock Springer-Verlag, Berlin-New York, 1979.

\bibitem{LN24}
E.~Lorist and Z.~Nieraeth.
\newblock Banach function spaces done right.
\newblock {\em Indag. Math., New Ser.}, 35(2):247--268, 2024.

\bibitem{LoVer}
E.~Lorist and M.C. Veraar.
\newblock Singular stochastic integral operators.
\newblock {\em Anal. PDE}, 14(5):1443--1507, 2021.

\bibitem{Nee}
J.M.A.M.~van Neerven.
\newblock {\em Functional analysis}, volume 201 of {\em Cambridge Studies in Advanced Mathematics}.
\newblock Cambridge University Press, Cambridge, Corrected printing, 2024.

\bibitem{NVWco}
J.M.A.M.~van Neerven, M.C. Veraar, and L.W. Weis.
\newblock Conditions for stochastic integrability in {UMD} {B}anach spaces.
\newblock In {\em Banach spaces and their applications in analysis (in honor of Nigel Kalton's 60th birthday)}, pages 127--146. De Gruyter Proceedings in Mathematics, De Gruyter, 2007.

\bibitem{NVW-UMD}
J.M.A.M.~van Neerven, M.C. Veraar, and L.W. Weis.
\newblock Stochastic integration in {UMD} {B}anach spaces.
\newblock {\em Ann. Probab.}, 35(4):1438--1478, 2007.

\bibitem{NVW-SMR}
J.M.A.M.~van Neerven, M.C. Veraar, and L.W. Weis.
\newblock Stochastic maximal {$L^p$}-regularity.
\newblock {\em Ann. Probab.}, 40(2):788--812, 2012.

\bibitem{NVW-Rbounded}
J.M.A.M.~van Neerven, M.C. Veraar, and L.W. Weis.
\newblock On the {$R$}-boundedness of stochastic convolution operators.
\newblock {\em Positivity}, 19(2):355--384, 2015.

\bibitem{NVW-survey}
J.M.A.M.~van Neerven, M.C. Veraar, and L.W. Weis.
\newblock Stochastic integration in {B}anach spaces---a survey.
\newblock In {\em Stochastic analysis: a series of lectures}, volume~68 of {\em Progr. Probab.}, pages 297--332. Birkh\"{a}user/Springer, Basel, 2015.

\bibitem{OgSh}
T.~Ogawa and S.~Shimizu.
\newblock Free boundary problems of the incompressible {Navier}-{Stokes} equations with non-flat initial surface in the critical {Besov} space.
\newblock {\em Math. Ann.}, 390(2):3155--3219, 2024.

\bibitem{pruss2016moving}
J.~Pr\"{u}ss and G.~Simonett.
\newblock {\em Moving interfaces and quasilinear parabolic evolution equations}, volume 105 of {\em Monographs in Mathematics}.
\newblock Birkh\"{a}user/Springer, 2016.

\bibitem{Weis-MathAnn}
L.W. Weis.
\newblock Operator-valued {F}ourier multiplier theorems and maximal {$L\sb p$}-regularity.
\newblock {\em Math. Ann.}, 319(4):735--758, 2001.

\end{thebibliography}

\end{document}